\documentclass[10pt,a4paper,oneside]{amsart}
\usepackage[normalem]{ulem}
\usepackage[utf8]{inputenc}
\usepackage{turnstile}
\usepackage{amsmath,amssymb,amsthm}
\usepackage{libertinus}
\usepackage{libertinust1math}
\usepackage[shortlabels]{enumitem}
\usepackage[]{geometry}
\usepackage[dvipsnames]{xcolor}
\usepackage{centernot}
\usepackage{comment}
\usepackage{bussproofs}
\usepackage{nicefrac}
\usepackage{proof,color}
\usepackage{natbib}
\usepackage{soul}

\usepackage[textwidth=1.5cm]{todonotes}

	\theoremstyle{definition}
		\newtheorem{remark}{Remark}

	\theoremstyle{plain}
        \newtheorem{thm}{Theorem}
	\newtheorem{dfn}[thm]{Definition}
		\newtheorem{lemma}[thm]{Lemma}
		\newtheorem{prop}[thm]{Proposition}
		\newtheorem{corollary}[thm]{Corollary}

		\newtheorem{op}{Open problem}

        \theoremstyle{definition}

        \newcommand{\T}{\mathrm{Tr}}

        \newcommand{\pk}{\mrm{PK}}

        \newcommand*\subdot[1]{\oalign{$#1$\cr\hfil.\hfil}}%
        \newcommand{\mc}[1]{\mathcal{#1}}
        \newcommand{\lnat}{\mathcal{L}_\mathbb{N}}
        \newcommand{\lt}{\mc{L}_\T}
        \newcommand{\mrm}{\mathrm}
        \newcommand{\vphi}{\varphi}
        \newcommand{\beq}{\begin{equation}}
        \newcommand{\eeq}{\end{equation}}
        \newcommand{\vsk}{\vDash_\mrm{sk}}
        \newcommand{\vssk}{\vDash_\mrm{ssk}}
        \newcommand{\ra}{\rightarrow}
        \newcommand{\Ra}{\Rightarrow}
        \newcommand{\pat}{\mrm{PAT}}
        \newcommand{\nat}{\mathbb{N}}
        \newcommand{\sth}{\;|\;}
        \newcommand{\corn}[1]{\ulcorner #1\urcorner}

        \newcommand{\lra}{\leftrightarrow}

        \newcommand{\indt}{\mrm{IT}}

        \newcommand{\ovl}{\overline}
        \newcommand{\val}[1]{#1^\circ}

\newbox\gnBoxA
\newdimen\gnCornerHgt
\setbox\gnBoxA=\hbox{$\ulcorner$}
\global\gnCornerHgt=\ht\gnBoxA
\newdimen\gnArgHgt

\def\gn #1{%
\setbox\gnBoxA=\hbox{$#1$}%
\gnArgHgt=\ht\gnBoxA%
\ifnum     \gnArgHgt<\gnCornerHgt \gnArgHgt=0pt%
\else \advance \gnArgHgt by -\gnCornerHgt%
\fi \raise\gnArgHgt\hbox{$\ulcorner$} \box\gnBoxA %
\raise\gnArgHgt\hbox{$\urcorner$}}

\title[Supervaluation-Style Truth Revisited]{
Supervaluation-Style Truth Revisited}
\author{Pablo Dopico \& Carlo Nicolai \& Johannes Stern}
\date{}

\begin{document}

\begin{abstract}
Supervaluational fixed-point semantics for truth cannot be axiomatized because of its recursion-theoretic complexity. Johannes Stern (\emph{Supervaluation-Style Truth Without Supervaluations}, Journal of Philosophical Logic, 2018) proposed a new strategy (supervaluational-style truth) to capture the essential aspects of the supervaluational evaluation schema whilst limiting its recursion-theoretic complexity, hence resulting in ($\nat$-categorical) axiomatizations. Unfortunately, as we show in the paper, this strategy was not fully realized in Stern's original work: in fact, we provide counterexamples to some of Stern's key claims. However,  we also vindicate Stern's project by providing different semantic incarnations of the idea and corresponding $\nat$-categorical axiomatizations. The results provide a deeper picture of the relationships between standard supervaluationism and supervaluational-style truth. 
\end{abstract}

\maketitle
\section{Introduction}\label{sec:intro}

{Formal truth theories can be divided into (at least) two camps: axiomatic and semantic theories of truth. The former introduce the truth predicate by specifying axioms and rules for the truth predicate while the latter define suitable models for the language with the truth predicate. These two different ways of introducing the truth predicate are not fully independent but deeply intertwined. As \cite{fischer_halbach_kriener_stern_2015} convincingly argue, axiomatic theories are frequently employed to motivate and in support of specific semantic theories. The same holds for the converse direction to the extent that specific axiomatic theories have been said to {\it axiomatize} certain semantic theories. But how is the latter claim to be understood? The problem is that semantic truth theories typically define the interpretation of truth predicate relative to the intended model of the syntax theory assumed -- the standard model of arithmetic if we arithmetize syntax. As a consequence, the set of sentences singled out by a given semantic theory will always form a non-axiomatizable set and the question arises on what grounds an axiomatic truth theory can be said to axiomatize a semantic theory of truth.
}



\cite{fischer_halbach_kriener_stern_2015} propose a number of criteria to this effect. Arguably, amongst these criteria, $\mathbb{N}$-categoricity stands out as particularly compelling: an axiomatic theory $\Sigma$ is an $\mathbb{N}$-categorical axiomatization of a semantic theory $\mrm{Th}_\mathfrak{S}$, where $\mathfrak{S}$ is a class of interpretations of the truth predicate defined over the standard model $\nat$, iff
\begin{align}\label{eq:ncat}(\mathbb{N},S)\vDash\Sigma&\text{ iff }S\in\mathfrak{S}.\end{align}

{The idea supporting the criterion is that, once we have fixed the intended interpretation of our syntax theory, then the axiomatic truth theory fixes the available interpretations of truth predicate, that is, the truth theory is categorical relative to the interpretation of the syntax theory. On the basis of the proposed criterion a firm link between suitable axiomatic and semantic truth theories can be established. For example,}
the theory $\mrm{KF}$ is an $\mathbb{N}$-categorical axiomatization of Kripke's theory of truth, that is, a set $S\subseteq\omega$ is an interpretation of the truth predicate of $\mrm{KF}$ over the standard model iff $S$ is a fixed point of the Strong Kleene jump $\mrm{SK}$:\footnote{The $\mrm{SK}$-jump can be defined as: for a set of codes of sentences $S$ and a formula $\vphi$, $\#\vphi\in \mrm{SK}(S)$ iff $(\mathbb{N}, S, S^-)\vDash_\mrm{sk} \vphi$, where $\vDash_\mrm{sk}$ is the Strong Kleene satisfaction relation (cf. \cite[p.192]{halbach_2014}), and $(\mathbb{N}, S, S^-)$ is the $\lt$-structure given by the standard model of arithmetic, $S$ as the extension of $\T$, and $S^-$ as the antiextension (see \S \ref{supervaluations}).}
\begin{align*}
(\mathbb{N},S)\vDash\mrm{KF}\text{ iff }\mrm{SK}(S)=S.\end{align*}
The $\nat$-categoricity has also been widely discussed in the recent literature, and it is often taken to be a necessary condition for axiomatizations of semantic truth theories.\footnote{See for instance \cite{rossi_2016}, \cite{halbach_nicolai_2018}, \cite{picollo2020a,picollo2020b},\cite{rosenblatt_2021},\cite{castaldo2023}, \cite{cani25}. } For these reasons, it is important to determine whether or not an axiomatic theory of truth is $\nat$-categorical with respect to a semantic theory of truth.
\footnote{$\nat$-categoricity is mainly cogent in the context of classical theories of truth where sentences $\vphi$ are in general not equivalent to $\T\corn{\vphi}$. When such an equivalence is in place -- e.g.~in the context of axiomatizations of Kripke's fixed-point semantics in a Strong Kleene setting --, $\nat$-categoricity may be trivially obtained.}


However, there are influential semantic theories in the literature that are not amenable to an $\mathbb{N}$-categorical axiomatization.
Indeed both the revision theory of truth (see, e.g., \cite{gupta_belnap_1993}) but also Kripke's theory of truth in its supervaluational variant are examples of such theories.\footnote{The impossibility to $\nat$-categorically axiomatize the revision theory of truth then permeates other truth theories based on revision rules, most notably Field's \cite{field_2008}.} The reason is that the definition of the respective class of (intended) models is too complicated in the recursion-theoretic sense: satisfaction -- the left side of the biconditional \eqref{eq:ncat} -- can be defined by a $\Delta^1_1$-formula of the language of second-order arithmetic. By contrast, the definition of the respective class of models (the right side of the biconditional) requires a $\Pi^1_1$-formula in the case of the supervaluation scheme -- and even more complicated formulae in the case of the revision theory or Field's theory. This observation entails that well-known theories intended to axiomatize supervaluational truth, notably the theory 
$\mrm{VF}$ from \cite{cantini_1990}, are not $\mathbb{N}$-categorical axiomatizations of the supervaluational version of Kripke's truth theory.

The theory $\mrm{VF}$ is an attractive theory and seems to embody many aspects of Kripke's supervaluational theory of truth: logical truths and, more generally, $\mrm{PAT}$-truths, can be proved to be true in $\mrm{VF}$ and the truth predicate of $\mrm{VF}$ is internally naive.\footnote{This means that ${\mrm{VF}}\vdash\T\gn\varphi\leftrightarrow\T\gn{\T\gn\varphi}$ for every sentence $\varphi$.}
It is then a natural question to ask whether there is a natural semantics that is amenable to an $\mathbb{N}$-categorical axiomatization by a $\mrm{VF}$-style theory, that is, a semantics that captures salient aspects of supervaluational semantics without resorting to the same recursion-theoretic complexity. If so, this would both be technically  interesting and, more importantly, establish the sought for link between  semantic and axiomatic truth predicates with distinctive supervaluational traits. 
%

In way of answering this question \cite{stern_2018} proposed the {\sc SSK-thesis}:

\smallskip
\begin{quote}
\textsc{SSK-thesis:} There is a version of Kripke's theory of truth that captures the salient aspects of the supervaluational schema whilst limiting its recursion-theoretic complexity, hence resulting in an $\nat$-categorical axiomatization. 
\end{quote}
\smallskip

\noindent
Indeed, \cite{stern_2018} puts forward such a semantics, \emph{supervaluation-style truth}, based on a satisfaction scheme that he calls $\mrm{SSK}$ -- for `supervaluational Strong Kleene'.  Moreover, Stern proposes an axiomatic truth theory, which he labels $\mrm{IT}$, and which he alleges to be an $\mathbb{N}$-categorical axiomatization of Kripke's truth theory based on the  scheme $\mrm{SSK}$. 

Of course, the \textsc{SSK thesis} is somewhat vague, since we haven't fully specified what the `salient aspects' of the supervaluational schema are. This can be clarified as follows. Let the supervaluation schemata $(\nat,S)\vDash_{s}\varphi$ be defined as 
\begin{equation}\label{eqn:def supervaluation scheme}
\forall S'\supseteq S(\Phi(S')\Rightarrow (\nat,S')\vDash\varphi),\end{equation}
where $\Phi$ is a subset of the powerset of $\mrm{Sent}_\T$ -- the set of all $\lt$ sentences -- and $\Phi(S')$ stands for $S'\in \Phi$.  Stern begins his analysis by noting that supervaluation schemata can also be understood as collecting classical consequences of the set of sentences Strong Kleene grounded in $S$ over the standard model. That is, (\ref{eqn:def supervaluation scheme}) is equivalent to: 
\begin{equation*}\tag{2*}\label{eqn:2star}
\exists G(S\vDash_\mrm{sk}G\,\&\,\forall S'\supseteq S(\Phi(S')\,\&\,(\nat,S')\vDash G\Rightarrow (\nat,S')\vDash\varphi).\footnotemark  
\end{equation*}

\footnotetext{$G$ is a set of sentences and by $(\nat,S)\vDash_e G$ we intend to convey that
$(\nat,S)\vDash_e\psi\text{ for all }\psi\in G.$}
\noindent In particular, the supervaluation schemata $\mrm{VB,VC,MC}$ discussed below can be viewed as collecting classical \emph{second-order} consequences of the set of Strong Kleene truths over the hypothesis $S$, modulo the constraints determined by each specific schema.\footnote{The correspondence between the two characterizations, i.e., between (\ref{eqn:def supervaluation scheme}) and (\ref{eqn:2star}), holds as long as $\Phi$ forces a precisification $S'$ to be consistent with the starting hypothesis $S$, that is, as long as $\Phi(S')$ entails that $\{\varphi\,|\,\neg\varphi\in S\}\cap S'=\emptyset$. The correspondence will thus not hold for the trivial admissibility condition $\Phi$ with $\Phi(S'):\lra S'=S'$ and thus fails for the scheme $\mrm{sv}$ discussed below.} 

Consequently, Stern observed that, instead of taking classical consequences of $G$ relative to a specific class of truth models over the standard model, one can simply consider the classical first-order consequences of the Strong Kleene-grounded sentences. This constitutes the SSK satisfaction scheme, and gives rise to what Stern called  \emph{supervaluation-style semantics}. The approach was intended to achieve a  significant reduction of the recursion-theoretic complexity and to give rise to an arithmetically definable jump.
Specifically, \cite{stern_2018} claimed that the fixed points of the $\mrm{SSK}$-scheme are the fixed points of a first-order arithmetical operator $\Theta$ (cf.~Definition \ref{arithmetical operator}), and that the $\mrm{SSK}$ schema has an $\mathbb{N}$-categorical axiomatization in the form of the theory $\mrm{IT}$. In addition, under these assumptions, he showed that the $\mrm{SSK}$-scheme has the same minimal fixed point as the supervaluation scheme.
%
The reconstruction of Stern's project also gives us a precise way of understanding the \textsc{SSK thesis}: the preservation of `the salient aspects' of the supervaluational scheme amounts to recovering all first-order penumbral truths of classical logic and arithmetic, that is, truths that have such status merely in virtue of the first-order logical and arithmetical relations that hold between the sentences of the language.\footnote{The term `penumbral truth' can be traced back to \cite{fine_1975}.} Such truths constitute indeed something that one would hope to retain when moving from second-order to first-order formulations of the supervaluational schema. 
{Summing up, following \cite{stern_2018}, supervaluation-style semantics \textit{first-orderize} the supervaluation schema in two distinct ways: first, via a \textit{first-order} arithmetical formula rather than second-order arithmetical formula; second, as the closure of Strong Kleene truths under their \textit{first-order} arithmetical consequences, as opposed to second-order.}


Unfortunately, as we show in this paper, the \textsc{SSK thesis} hasn't been successfully realized by \cite{stern_2018} since some of Stern's reasoning was flawed. Specifically:
\begin{itemize}
\item the fixed points of the operator associated with the $\mrm{SSK}$-schema and $\Theta$ constitute two distinct classes of fixed points;
\item the fixed points of the $\mrm{SSK}$-schema are \emph{not} closed under the $\omega$-rule;
\item as a consequence, the  minimal fixed point of the $\mrm{SSK}$-schema and the supervaluation scheme differ;
\item the theory $\mrm{IT}$ is not $\nat$-categorical with respect to $\mrm{SSK}$-fixed points. 
\end{itemize}

However, not all is lost and, as we show, {the \textsc{SSK-thesis} can be vindicated}, albeit in a less uniform way than argued in \cite{stern_2018}. In fact, since the fixed points of the $\mrm{SSK}$ schema and $\Theta$ come apart, we end up with different semantic constructions -- each of them recovering penumbral truths of classical logic and arithmetic -- that enjoy $\nat$-categorical axiomatizations in virtue of their lower recursion-theoretic complexity.\footnote{We recall that, because of the higher recursion-theoretic complexity of supervaluational fixed-point semantics, $\nat$-categorical axiomatizations will not be available for the latter. In some sense it is impossible to equip supervaluational fixed-point semantics with a fitting proof theory. This shows, for example, in the fact that the models of $\mrm{VF}$ are not just the supervaluational fixed points but, in addition, fixed points of $\Theta$ or even the set of stable truth of the Herzberger sequence over the empty set \cite[\S96]{cantini_1996}. It thus seems difficult to associate $\mrm{VF}$ with supervaluational fixed points in some strong sense. } 
The choice forced upon us essentially {corresponds to the two different ways of first-orderizing mentioned above}. We can {either focus on preserving the idea of taking first-order consequences of strong Kleene truths and} stick to the well-motivated $\mrm{SSK}$ schema. {However, in this case we} lose the desirable closure under $\omega$-logic. {As a consequence, the universal quantifier will not commute with the truth predicate in $\nat$-categorical axiomatizations of the $\mrm{SSK}$ schema.}

Instead, we can focus {on the alternative form of first-orderizing and} adopt a more `standard' arithmetical jump with familiar closure conditions. {However, the arithmetical jump will, in contrast to the $\mrm{SSK}$ schema and its associated operator, ``merely'' collect consequences of a given truth set/hypothesis rather than submitting the hypothesis to semantic scrutiny. {As a consequence, we may obtain fixed points which have sentences as their members whose ``truth'' is not supported by any rationale. For example, disjunctions may turn out ``true'' in this sense even though they are neither logical or arithmetical truths nor is either of their disjuncts true. This lack of semantic scrutiny} 
 undermines the philosophical and semantic motivation underlying supervaluational-style semantics} and leads to some further formal features that seem undesirable from a supervaluational perspective (e.g.~failure of what will be called `classical soundness'). {One principal upshot of the paper is that these two different strategies of first-orderizing the supervaluation schema will come apart.}


In the paper we will study the classes of models given by fixed points of $\mrm{SSK}$, $\Theta$, and  $\Theta^*$ -- a variant of $\Theta$ equipped with additional closure properties. We will also provide $\nat$-categorical axiomatizations for each of them -- as we have seen, a highly desirable feature. 
After describing some preliminaries in Section \ref{sec:preliminaries} and Section \ref{sec:supervaluation-style-truth}, we isolate some gaps and problems in the previous attempt to realize the \textsc{SSK thesis} by \cite{stern_2018} in Section \ref{sec:problems}. These include $\mrm{SSK}$'s failure to be closed under $\omega$-logic.  Section \ref{sec:n-cat ssk} provides two $\nat$-categorical axiomatizations of $\mrm{SSK}$, that we will dub $\mrm{PK}$ and $\mrm{PK}^+$. We then move to the study of $\Theta$ and its variation $\Theta^*$: Section \ref{sec:n-cat theta} presents $\nat$-categorical axiomatizations for (the fixed points of) each operator, and establishes lower-bounds for their proof-theoretic strength. 
Section \ref{sec:relssksup} draws connections between supervaluation-style fixed-point semantics and the standard supervaluational approaches: although the minimal fixed points of $\mrm{VB}$ (a class of supervaluational fixed points that imposes a weak consistency condition), $\Theta$ and $\Theta^*$ coincide, this is not true for all fixed points. In fact, every $\mrm{VB}$ fixed point is a $\Theta^*$, and hence a $\Theta$ fixed point, but the converse provably fails. We conclude by assessing the results presented in the paper along some key dimensions.

\section{Preliminaries and notation}\label{sec:preliminaries}

\subsection{Language, Theories, Coding}\label{subsec:notation}

Our notation follows, for the most part, the conventions in \cite{halbach_2014}. We will be working with languages whose logical symbols are $\neg, \vee, \wedge, \forall, \exists$. We also write $\varphi \rightarrow \psi$ as an abbreviation for $\neg \varphi \vee \psi$. As a default practice, our base language $\lnat$ is a definitional extension of the language of Peano Arithmetic ($\mrm{PA}$) with finitely many function symbols for primitive recursive functions. We assume a standard formalization of the syntax of first-order languages as it is for instance carried out in \cite{hapu98}. The expression $\Bar{n}$ stands for the numeral of the number $n$ (although we omit the bar for specific numbers). We write $\# \varphi$ for the code or Gödel number of $\varphi$, and $\ulcorner \varphi\urcorner$ for the numeral of that code. $\lnat$ is assumed to contain a finite set of function symbols that will stand for certain primitive recursive operations. In particular, we want to include function symbols for some primitive recursive syntactic operations on expressions of the language under consideration; for that, we use the traditional subdot notation. For example, $\subdot \neg$ is the function symbol for the primitive recursive operation that, when inputted the code of a formula, yields the code of its negation. The same applies to $\subdot \vee, \subdot \forall, \subdot \wedge, \subdot \exists$. 
We assume a function symbol for the substitution function, and write $x (t/v)$ for the result of formally substituting $v$ with $t$ in $x$; $\ulcorner \varphi(\dot x)\urcorner$ abbreviates $\ulcorner \varphi (v)\urcorner (\mrm{num}(x)/\ulcorner v\urcorner)$, for $x$ a term variable (and provided $\varphi$ only has one free variable) and $\mrm{num}(x)$ the function symbol corresponding to the operation of sending a number to its numeral. Furthermore, we avail ourselves to the notation $t^\circ$ for the result of applying to a term $t$ the evaluation function (which outputs the value of the inputted term). Note that $\cdot^\circ$ is an abbreviation for a formula, and not a term of the language.  The development of syntax just mentioned can be comfortably carried out in Peano arithmetic $(\mrm{PA})$, which will be our background theory of syntax.

For the formulation of the theories of truth, we will (mostly) be concerned with the extension of $\lnat$ with a (unary) truth predicate $\T$. $\pat$ is then the theory formulated in the language $\lt$ consisting of the axioms of $\mrm{PA}$ with induction extended to the whole language. We will then write $\mrm{Var}_\T(x)$ for the formula representing the set of (codes of) variables, $\mrm{CT}_\T (x)$ for the formula representing the set of (codes of) closed terms of ${\lt}$, $\mathrm{For}_\T (x)$ ($\mathrm{For}_\T(x,y_0,...,y_n)$), for the formula representing the set of all formulae (formulae with $y_0,...y_n$ free) of ${\lt}$, and $\mathrm{Sent}_\T(x)$, for the formula representing the set of all sentences of ${\lt}$. We also use $\mrm{Var}_{\T}$, $\mrm{CT}_{\T}$, $\mrm{For}_{\T}$, and $\mrm{Sent}_{\T}$ to stand for the sets corresponding to these formulae.  We will occasionally omit reference to ${\lt}$ when this is clear from the context. We often quantify over variables $s,t$ to refer to codes of closed terms, so that $\forall t\,\vphi(t)$ is short for $\forall x (\mrm{CT}_\T (x)\ra \vphi(x))$.

We assume an appropriate ordinal  notation system OT up to $\Gamma_0$, and the ordinal less-than relation $<$ on $\Gamma_0$ (for more detailes, see, e.g., \cite[Ch.3]{pohlers_2009}). For simplicity, we identify each ordinal number with its notation. We use $0$ as the ordinal number, the ordinal sum $\alpha + \beta$, and the Veblen function $\vphi_{\alpha}\beta$. 
%
Given a language $\mathcal{L}$, an ordinal number $\alpha < \Gamma_0$, and an $\mathcal{L}$-formula $A$, transfinite induction for $A$ up to $\alpha$ is defined as the formula: 
\[
\mrm{TI}(\alpha, A):= \forall \beta (\forall \gamma < \beta A(\gamma) \to A(\beta)) \to \forall \beta < \alpha A(\beta).
\]
In line with the latter, the schema $\mrm{TI}_{\mathcal{L}}(<\alpha)$ is defined to be the set $\{ \mrm{TI}(\beta, A) \sth \beta < \alpha \ \& \ A \in \mathcal{L} \}$. 

Finally, if $S$ is a theory, its proof-theoretic ordinal (denoted $\vert S\vert$) is the supremum of all ordinals that are provably recursive in $S$, i.e., such that there is a recursive well-ordering $\langle A, \prec\rangle$ of order-type $\alpha$ which is provably well-ordered in $S$. For a natural choice of an ordinal notation system, this coincides with the supremum of all $\alpha$ such that $S\vdash \mrm{TI}_{\lnat}(<\alpha)$. As is well-known, it is generally the case that, for such an ordinal $\alpha$, $S$ and $\mrm{PRA}+\mrm{TI}_{\lnat}(<\alpha)$ have the same $\Pi^0_2$ consequences in the language $\lnat$. If $T$, $S$ are theories, we write $\vert T\vert \geqq \vert S\vert$ to indicate that the proof-theoretic ordinal of $T$ is equal or greater than the proof-theoretic ordinal of $S$. Similarly, $\vert T\vert \equiv \vert S\vert$ is defined as $\vert T\vert \geqq \vert S\vert$ and $\vert S\vert \geqq \vert T\vert$.

\subsection{Supervaluations}\label{supervaluations}

In his seminal \cite{kripke_1975}, Kripke proposes to determine the extension of the truth predicate via a fixed-point construction. Famously, this construction is carried out in Strong Kleene (SK) logic. However, SK fails to capture many \textit{penumbral} truths, that is, truths that hold merely in virtue of the logical relations occurring between sentences of a language. An example of such a sentence is $\lambda\vee\neg \lambda$, for $\lambda$ the Liar sentence. As a remedy, Kripke in his paper also suggests to run the construction with some form of supervaluational logic. This will be our departure point. For $X\subseteq \mrm{Sent}_\T$, we define
\begin{align*}
    \mrm{CON}(X)&:=\neg \exists \vphi (\{\#\vphi,\#\neg\vphi\}\subseteq X);\\
    \mrm{MXC}(X)&:= \neg \exists \psi(\{\#\psi,\#\neg\psi\}\subseteq X\land (\forall \vphi(\#\vphi\notin X \ra \#\neg \vphi\in X))).
\end{align*}
$\mrm{MXC}(X)$ expresses that $X$ is a maximally consistent set of $\lt$-sentences. 

For a satisfaction relation $e$, we will write $X \vDash_{e} A$ to abbreviate $(\nat,X, X^-)\vDash_{{e}} A$, where $X\subseteq \mrm{Sent}_\T$ is the relevant interpretation of the truth predicate and $X^-$ is defined as $\{\#\vphi \sth \#\neg\vphi \in X\}\cup \{n\in\omega\sth \nat \vDash \neg \mrm{Sent}_\T(\bar n)\}$. Similarly, for the classical satisfaction relation ($\vDash$), $X\vDash A$ just stands for $(\nat, X)\vDash A$. In the literature, one often distinguishes the following supervaluational satisfaction relations:
\begin{align}
\label{eq:sva}X\vDash_\mrm{sv} \varphi & \text{ iff }\forall X' (X'\supseteq X \Rightarrow X' \vDash \varphi)\\
\label{eq:svb} X\vDash_\mrm{vb} \varphi &\text{ iff } \forall X' (X'\supseteq X \, \&\,  X' \cap X^- =\varnothing \Rightarrow X' \vDash \varphi)\\
    X\vDash_\mrm{vc} \varphi &\text{ iff } \forall X' (X'\supseteq X\, \&\,  \mrm{CON}(X') \Rightarrow X' \vDash \varphi)\\
  \label{eq:svd}   X\vDash_\mrm{mc} \varphi &\text{ iff } \forall X' (X'\supseteq X\, \&\, \mrm{MXC}(X') \Rightarrow X' \vDash \varphi)
\end{align}

To avoid triviality, and following the truth-theoretic literature on supervaluationism, we are mainly interested in consistent set of sentences. For $\mrm{CON}(X)$ and $j\in \{\mrm{sv, vb, vc, mc}\}$, we let: 
\beq
J(X):=\{\#\varphi\sth \, X\vDash_j \varphi\}
\eeq

\noindent $J(\cdot)$, for $J\in \{\mrm{SV, VB, VC, MC}\}$, is the \emph{J-supervaluational jump}. It can be checked that, for arbitrary $X\subseteq \mrm{Sent}_\T$, the jump is monotone. It then follows by the theory of positive inductive definitions (see e.g. \cite[Ch.1]{moschovakis_1974}, \cite[Ch.5]{mcgee_1991}) that $J(\cdot)$ has a fixed point, and in particular a minimal one obtained by iterating the operator over the empty set. We denote the minimal fixed point of a supervaluational jump (or, in general, of any Kripke-style presented in the paper) $J(\cdot)$ as $\mrm{I}_{J}$.

It is worth noting that, to establish the existence of non-trivial fixed points, we need a further property to guarantee that at every stage of the inductive definition we can find an admissible precisification. Otherwise triviality ensues. One property that guarantees the existence of non-trivial fixed points is to require the admissibility condition to be compact on the set of consistent interpretations: let $\Phi$ be, again, a subset of $\mc{P}(\mrm{Sent}_\T)$, let $X:=\{S\,|\,\mrm{CON}(S)\}$, let $\Phi^S:=\{S'\,|\,S\subseteq S'\,\&\,\Phi(S')\}$, and, for $Y\subseteq X$, set $\Phi(Y):=\{\Phi^S\,|\,S\in Y\}$. Then $\Phi$ is \emph{compact on $X$} iff for all $Y\subseteq X$:
$$\text{if }\Phi^{S_1}\cap\ldots\cap\Phi^{S_n}\neq\emptyset\text{ for all }S_1,\ldots S_n\in Y\text{ and all }n\in\omega,\text{ then }\bigcap\Phi(Y)\neq\emptyset.$$
The admissibility condition associated with the schemes $\mrm{SV, VB, VC, MC}$ are compact in that sense. 


\section{Supervaluation-style truth}\label{sec:supervaluation-style-truth}
While Kripke's fixed-point semantics over supervaluational logics succeeds in recovering all penumbral truths, the theory is not free from objections. As argued in \cite{stern_2018} supervaluational truth fails to be compositional and, because of its recursion-theoretic complexity, is not $\nat$-categorically axiomatizable. Accordingly, Stern develops a new theory of truth, which he calls \emph{supervaluational-style truth}, aimed at addressing some of these shortcomings. 

As we have seen, to address such worries Stern  notes that supervaluation \emph{satisfaction} involves recognizable \emph{Strong-Kleene} and \emph{classical} components. Stern studies a monotone evaluation schema isolating these two elements of supervaluational satisfaction:
\beq\label{eq:sskst}
X\vssk \varphi \text{ iff } \exists \psi(X\vsk\psi \text{ and } \pat\vdash \psi\rightarrow \varphi)
\eeq 
In the definition of $\vssk$, $\vsk$ is the familiar definition of satisfaction in Strong Kleene logic for $\lt$. 
For $X\subseteq \omega$ consistent, the \emph{Strong-Kleene supervaluational jump} is then:
\beq
\mrm{SSK}(X):=\{\#\varphi\sth \, X\vssk \varphi\}.
\eeq
The jump is monotone, and therefore will have fixed points, including a minimal one. 

To study the Strong-Kleene supervaluational jump, Stern also introduces a related jump based on a positive inductive definition in $\lnat$ (with parameters). 
\begin{dfn}\label{arithmetical operator}
Let $\xi(x,X)$ abbreviate the familiar formula of $\lnat$ employed to characterize Strong Kleene truth: 
\begin{align*}
    &  x\in \mrm{True}_0\\
        &\vee \exists y (x=\subdot\neg\subdot\neg y \wedge y \in X)\\
		&\vee  \exists y,z(x=(y\subdot{\vee} z) \wedge (y\in X\vee z\in X))\\
        &\vee  \exists y,z(x=\subdot\neg(y\subdot{\vee} z) \wedge (\subdot\neg y\in X\wedge \subdot\neg z\in X))\\
		&\vee  \exists y,z(x=(y\subdot{\wedge} z) \wedge (y\in X\wedge z\in X))\\
        &\vee  \exists y,z(x=\subdot\neg(y\subdot{\wedge} z) \wedge (\subdot\neg y\in X\vee \subdot\neg z\in X))\\
		&\vee  \exists y,v(x=(\subdot{\forall} vy) \wedge \forall z(y (z/v)\in X))\\
        &\vee  \exists y,v(x=\subdot \neg (\subdot{\forall} vy) \wedge \exists z(\subdot \neg y (z/v)\in X))\\
		&\vee  \exists y,v(x=(\subdot{\exists} vy) \wedge \exists z(y (z/v)\in X))\\
        &\vee  \exists y,v(x=\subdot \neg (\subdot{\exists} vy) \wedge \forall z(\subdot \neg y (z/v)\in X))\\
		&\vee  \exists t(x=\subdot \T (t) \wedge t^\circ \in X)\\
		&\vee  \exists t(x=\subdot \neg \subdot \T (t) \wedge (\subdot{\neg}t^\circ)\in X\vee \neg \mrm{Sent}(t^\circ))
\end{align*}
In the definition, $\mrm{True}_0$ is the set of all codes of true arithmetic literals. Then define 
\[
\theta(x,X):\lra \exists y (\xi(y,X)\wedge \mathrm{Pr}_{\mathrm{PAT}} (y\subdot \ra x)).
\]
Finally, the operator $\Theta: \mc{P}(\omega)\to \mc{P}(\omega) $ is defined as 
\[
\Theta(S):=\{ n\in \omega \sth \nat \vDash \theta(x,X) [n,S]\}.
\]
\noindent where $\Phi(x,X) [m,M]$ stands for the result of replacing, in the formula $\Phi(x,X)$, the variable $x$ for the numeral of $m$, and $X$ for a predicate $S^*$ that extends $\lnat$ and whose interpretation is the set $S$. 
\end{dfn}

Stern provides arguments to show that (i) the fixed points of $\mrm{SSK}$ and $\Theta$ coincide and, making crucial use of (i), that (ii) there is an $\nat$-categorical (in the sense of \cite{fischer_halbach_kriener_stern_2015}) axiomatization of those fixed points. To achieve (ii), Stern introduces the following theory. 
\begin{dfn}\label{dn:ind truth}
The theory $\indt$ extends $\pat$ with the following axioms: 

\begin{enumerate}
	\item[$\mrm{IT1}$] $\forall s,t(\T(s\subdot{=}t)\leftrightarrow s^\circ=t^\circ) \wedge \forall s,t(\T(s\subdot{\neq} t)\leftrightarrow s^\circ\neq t^\circ)$
	\item[$\mrm{IT2}$] $\forall x,y(\mrm{Sent}(x\subdot{\wedge}y)\rightarrow (\T (x)\wedge \T (y)\rightarrow \T(x\subdot{\wedge}y)))$
	\item[$\mrm{IT3}$] $\forall x,y(\mrm{Sent}(x\subdot{\vee}y)\rightarrow (\T (x)\vee \T (y)\vee \exists z(\xi(z, \T)\wedge\mrm{Pr}_{\pat}(z\subdot \rightarrow x\subdot{\vee}y ))\leftrightarrow \T(x\subdot{\vee}y)))$
	\item[$\mrm{IT4}$]$\forall x,v (\mrm{Sent}(\subdot{\forall}vx)\rightarrow (\forall t\T(x (t/v))\rightarrow\T(\subdot{\forall} v x)))$
	\item[$\mrm{IT5}$]$\forall x,v(\mrm{Sent}(\subdot{\exists}vx)\rightarrow (\exists t\T(x(t/v))\vee \exists w(\xi(w, \T)\wedge\mrm{Pr}_{\pat}(w\subdot{\rightarrow} \subdot{\exists}vx ))\leftrightarrow \T( \subdot{\exists}vx)))$
	\item[$\mrm{IT6}$]$\forall x(\T (x)\leftrightarrow \T(\subdot \T x))$
	\item[$\mrm{IT7}$] $\forall x(\T(\subdot{\neg} x) \vee \neg \mrm{Sent}(x)\leftrightarrow \T (\subdot \neg \subdot \T x) )$
	\item[$\mrm{IT8}$]$\forall x,y(\T (x)\wedge \mrm{Pr}_{\pat}(x\subdot{\rightarrow}y)\rightarrow \T (y))$ 
	\item[$\mrm{IT9}$]$\forall x(\T (\subdot{\neg}x)\rightarrow \neg \T (x))$
	\item[$\mrm{IT10}$]$\T\ulcorner \forall x(\T ( x)\rightarrow \mrm{Sent}( x))\urcorner$
	\item[$\T\text{-}\mrm{Out}$] $\forall t_1,...,t_n(\T\ulcorner \vphi(\subdot t_1,...,\subdot t_n)\urcorner \ra \vphi(t_1,...,t_n)$ for every formula $\vphi(t_1,...,t_n)$ of $\lt$.	
\end{enumerate}	
\end{dfn}
As it turns out the argument in support of (i) was flawed and, as a consequence, neither (i) nor (ii) can be achieved.

\section{Supervaluational-style truth: some problems}\label{sec:problems}

Stern tried to establish that the fixed points of SSK and $\Theta$ coincide, i.e., that $S=\Theta(S) \Leftrightarrow S=\mrm{SSK}(S)$. However, the proof relied on the mistaken assumption that $\Theta(S)\subseteq \mrm{SSK}(S)$ for any consistent set of formulas $S$ (see \cite[Lemma 3.4]{stern_2018}). This fails for the following reason. We say that a set $X\subseteq \omega$ is \emph{sound} with respect to some operator $\mc{J}:\mc{P}(\omega)\longrightarrow \mc{P}(\omega)$ when $X\subseteq \mc{J}(X)$.\footnote{The property of soundness is called \emph{inclusivity} in \cite{Hinman1978}.} Then, the above failure is to be located in the fact that \textit{any} set of formulas is sound with respect to $\Theta$, and this does not happen for $\mrm{SSK}$---not even when the set of formulas is consistent. 

\begin{prop}\label{theta is always sound}
Let $S$ be a any set of codes of formulae of $\lt$. Then, $S\subseteq \Theta(S)$. 
\end{prop}
\begin{proof}
If $\#\vphi\in S$, then $\nat \vDash \xi(\corn{\vphi\wedge\vphi}, S)$ and, since $\vphi\wedge\vphi\ra \vphi$ is provable in $\pat$, $\#\vphi \in \Theta(S)$. 
\end{proof}
Proposition \ref{theta is always sound} immediately entails that iterations of $\Theta$ over any $X\subseteq \omega$ reach a fixed point. These fixed points may turn out to be rather ill-behaved -- even if we focus on the non-trivial, consistent fixed points:
\begin{prop}\label{failures of theta}
There are formulae $\vphi, \psi$ of $\lt$ and $S$ a consistent set of codes of sentences such that: 

\begin{enumerate}
    \item $S=\Theta(S)$ and $\#\T\corn{\vphi} \in S$ but $\#\vphi \notin S$.
    \item $S=\Theta(S)$ and $\#\neg \T\corn{\vphi} \in S$ but $\#\neg \vphi \notin S$.
    \item $S=\Theta(S)$ and $\#\vphi \vee\psi \in S$ but $\#\vphi\notin S$ and $\#\psi\notin S$ and there is no $\chi$ such $S\vDash_{\mrm{sk}}\chi$ and $\mrm{PAT}\vdash \chi\ra\vphi\vee\psi$.
    \item $S=\Theta(S)$ and $\#\exists v \vphi \in S$ but $\#\vphi(\bar n)\notin S$ for all $n\in\omega$ and and there is no $\chi$ such $S\vDash_{\mrm{sk}}\chi$ and $\mrm{PAT}\vdash \chi \ra \exists v\vphi$.
\end{enumerate}
\end{prop}
\begin{proof}
First of all, we say that a formula $\tau$ (or $\tau_i$) is a truth-teller whenever it is $\mrm{PAT}$-provably equivalent to $\T\corn\tau$ (resp., $\T\corn{\tau_i}$). Now, the proof runs as follows.
\begin{itemize}
\item For 1: consider the (clearly consistent) fixed point that is obtained by iterating $\Theta$ over the set $X=\{ \#\T \corn{\neg \tau}\}$, for $\tau$ a truth-teller. Since $\Theta$ is also sound for $X$ (by Prop. \ref{theta is always sound}), $\#\T \corn{\neg \tau}$ is in the fixed point obtained from $X$. But $\#\neg\tau$ is not, since this can only be added at some stage if $\#\neg\T\corn\tau$ has been added in the previous stage.
\item For 2: consider the fixed point that is obtained by iterating $\Theta$ over the set $X=\{ \#\neg \T \corn{\neg \tau}\}$, for $\tau$ a truth-teller. Similarly to (1), $\#\neg \T \corn{\neg \tau}$ is in the fixed-point, but $\#\neg\tau$ is not.
\item For 3: consider the fixed point $S$ that is obtained by iterating $\Theta$ over the set $X=\{\#\tau_0\vee\tau_1\}$, for $\tau_0$ and $\tau_1$ truth-tellers; clearly, $\#\tau_0\notin S$ and $\#\tau_1\notin S$, as in no stage of the construction are either $\#\T\corn{\tau_0}$ or $\#\T\corn{\tau_1}$ added. Now suppose there is a $\chi$ such that $S\vDash_{\mrm{sk}}\chi$ and $\mrm{PAT}\vdash\chi\ra\tau_0\vee\tau_1$. We obtain $S\vDash\chi$\footnote{We note that for consistent sets of sentences  $S$ we have that $$\text{if }S\vDash_\mrm{sk}\varphi,\text{ then }S\vDash\varphi,$$ for all sentences $\varphi$.} and $S\vDash\chi\ra\tau_0\vee\tau_1$ and thus $S\vDash\tau_0\vee\tau_1$. Hence $S\vDash\tau_0$ or $S\vDash\tau_1$; WLOG, assume $S\vDash\tau_0$. Since $\tau_0$ is provably equivalent to $\T\corn{\tau_0}$, this entails $\#\tau_0\in S$. Contradiction. 
\item For 4: consider the fixed point that is obtained by iterating $\Theta$ over the set $X=\{ \#\exists x (\tau\wedge x\geq 0)\}$, and apply the reasoning for (3) above.
\end{itemize}
\end{proof}

The proposition entails that the key property of classical soundness (cf.~Def.~\ref{dfn:classound_downnwards_closure}) does not obtain for all $\Theta$ fixed-points. Classical soundness will play a prominent role in what follows. 

\begin{corollary}\label{failure of classical soundness}
There are set of codes of sentences $S$ such that $S\subseteq \Theta(S)$ and $\#\vphi\in S$ but $S\nvDash \vphi$. In particular, there are fixed points $S = \Theta(S)$ such that $\#\vphi\in S$ but $S\nvDash \vphi $.
\end{corollary}

Proposition \ref{theta is always sound} fails for SSK. By taking, for instance, $S=\{\#\tau_0\vee\tau_1\}$, where $\tau_0$ and $\tau_1$ are truth-tellers, one can immediately see that $S\nsubseteq \mrm{SSK}(S)$.
However, although none of the problems above apply to $\mrm{SSK}$ fixed points, the latter are not very attractive either, as they are not closed under $\omega$-logic, a basic feature of SK fixed points. That is, we show:

\begin{prop}\label{failure of omega logic}
There are fixed points $X$ of $\mrm{SSK}$ and a formula $\vphi(x)$ of $\lt$ such that $\# \vphi(\bar{n})\in X$ for any $n\in \omega$, but $\#\forall y\, \vphi\notin X$. 
\end{prop}

To prove Proposition \ref{failure of omega logic}, we employ a specific instance of transfinite induction up to $\varepsilon_0$ involving a parametrized truth teller formula $\tau(x)$, which we label $\mrm{TI}(\tau,\varepsilon_0)$. We show that, while each instance of transfinite induction for $\tau(x)$ up to $\alpha<\varepsilon_0$ is in the minimal fixed point $\mrm{I}_\mrm{ssk}$ (because they are provable in $\mrm{PAT}$), $\mrm{TI}(\tau,\varepsilon_0)$ cannot be. To establish the last step, we need to study the $\mrm{SSK}$ jump in some more detail. 

First, we notice that 
\beq
\mrm{SSK}(X)= \{ A \in \lt \sth \pat, \mrm{sk}(X) \vdash A\},
\eeq
where $\mrm{sk}(X)=\{B \in \lt \sth X\vDash_\mrm{sk} B\}$. Therefore, by the finiteness of $\pat$ proofs, 
\beq\label{eq:np6a}
A\in\mrm{SSK}(X) \text{ iff } \pat, B \vdash A
\eeq
for some $B\in \mrm{sk}(X)$. 

For our purposes, it is also useful to note that the set $\mrm{sk}(X)$ can be characterized proof-theoretically. We write $\Vdash^X A$ to express that there is a cut-free proof of $A$ in a one-sided sequent calculus whose initial sequents are expressions of the form $\Gamma, \T t$ for $t^\nat \in X$, and $\Gamma,\neg \T t$ for $\neg t^\nat \in X$, and $\Gamma,B$ for $B$ a true arithmetical literal. Therefore we have 
\beq\label{eq:np6b}
A \in \mrm{sk}(X) \text{ iff } \Vdash^X A. 
\eeq 

Note that, being cut-free, and given the axioms of the calculus, $\Vdash^X A,B$ holds iff $\Vdash^X A$ or $\Vdash^X B$. We are now ready to prove Proposition \ref{failure of omega logic}. 

\begin{proof}
By the diagonal lemma, we can obtained a parametrized truth-teller:
\beq
    \tau(x):\lra \T(\corn{\tau(x)}(\mrm{num}(x)/\corn{x})).
\eeq
It is important to note that, as is well-known, for any term $t$, $\tau(t)\notin \mrm{I}_\mrm{ssk}$, and hence $\tau(t)$ cannot be included in any stage of the construction of the minimal fixed point. Assuming the ordinal notation system introduced in \ref{subsec:notation}, let 
\begin{align*}
\mrm{TI}(\tau,\alpha)&:\lra \mrm{Prog}(\tau) \ra (\forall x <\alpha)\,\tau(x),
\end{align*}

Since $\pat\vdash \mrm{TI}(\tau,\alpha)$ for each $\alpha<\varepsilon_0$, each such $\mrm{TI}(\tau,\alpha)$ is in $\mrm{I}_\mrm{ssk}$, the minimal fixed point of $\mrm{SSK}$. Seeking a contradiction, we now assume that $\mrm{TI}(\tau,\varepsilon_0)$ is also in $\mrm{I}_\mrm{ssk}$. Therefore, there is a least $\alpha$ such that $\mrm{TI}(\tau,\varepsilon_0)\in \mrm{I}^\alpha_\mrm{ssk}$. Let $\mrm{I}^\alpha_\mrm{ssk}=Y$. By \eqref{eq:np6a} and \eqref{eq:np6b}, 
\beq\label{eq:np6c}
\pat,B \vdash \mrm{TI}(\tau,\varepsilon_0) 
\eeq
\text{ and } 
\beq\label{eq:np6d}
\Vdash^{Y} B.
\eeq

Now let $\mrm{PAT}^{Y}_\omega$ the system obtained by extending a standard infinitary version of $\mrm{PA}$ (with $\T$ in logical and non-logical schemata) with the initial sequents $\Gamma, \T t$ for $t^\nat \in {Y}$, and $\Gamma,\neg \T t$ for $\subdot \neg t^\nat \in {Y}$---which we shall call $Y$-initial sequents.
It follows that the proofs in \eqref{eq:np6c} and \eqref{eq:np6d} can be embedded in suitable $\mrm{PAT}^{Y}_\omega$ proofs. Specifically, by the cut-free nature of $\Vdash^{Y} B$, we can simply reproduce the derivation from \eqref{eq:np6d} within $\mrm{PAT}^{Y}_\omega$  to obtain a cut free derivation of $B$ in  $\mrm{PAT}^{Y}_\omega$ \emph{using only SK rules and $Y$ initial sequents}---what we shall call a SK-proof from $Y$:\footnote{As customary, $\vdash^\alpha_n$ expresses the existece of a proof with length $\leq \alpha$ and cut-rank $n$.}
\beq
\mrm{PAT}^{Y}_\omega \vdash^\alpha_0 B\;\text{ with $\alpha<\omega_1^\mrm{ck}$.}
\eeq 

On the other hand, we can apply a standard embedding strategy  to \eqref{eq:np6c} to obtain a pure $\pat\omega$ proof of $\neg B,\mrm{TI}(\tau,\varepsilon_0) $ -- that is, without occurrences of the initial sequents $\Gamma, \T t$ for $t^\nat \in {Y}$, and $\Gamma,\neg \T t$ for $\subdot \neg t^\nat \in {Y}$. By a standard argument, the length of this proof can be taken to be $<\omega\cdot 2$. By applying a standard cut elimination strategy for $\mrm{PA}\omega$, we then obtain that $\neg B,\mrm{TI}(\tau,\varepsilon_0) $ is provable in $\mrm{PAT}\omega$ with length $<\varepsilon_0$.

{Therefore, we have
\[
\AxiomC{$\mc{D}_0\vdash^\alpha_0 B,\mrm{TI}(\tau,\varepsilon_0)$}
\AxiomC{$\mc{D}_1\vdash^\beta_0\neg B,\mrm{TI}(\tau,\varepsilon_0)$}
        \BinaryInfC{$\vdash \mrm{TI}(\tau,\varepsilon_0)$}
\DisplayProof
\]
where $\mc{D}_0$ and $\mc{D}_1$ are both  cut-free $\mrm{PAT}^{Y}_\omega$ proofs, but $\mc{D}_0$ is a pure $\mrm{SK}$-proof from $Y$, and $\mc{D}_1$ a pure $\pat_\omega$-proof,  and $\beta<\varepsilon_0$. It can then be shown, by a standard reduction strategy proved by induction on the length of the proof where this occurrence of cut occurs, that the formula $B$ can be taken to be either an arithmetical literal or a truth literal ($\T t$ or $\neg \T t$). We treat the two cases separately.}

\emph{Case 1:} $B$ is of the form $s=t$ or $s\neq t$. Suppose the former. then

\begin{equation*}
    \pat^\omega \sststile{0}{\beta} \mrm{TI}(\varepsilon_0, \tau), s\neq t.
\end{equation*}
with $\beta<\varepsilon_0$. Now, either $s=t$ is false, or $s\neq t$ is false. If $s\neq t$ is false, standard $\pat^\omega$ reasoning yields:
\begin{equation*}
    \pat_\omega \sststile{0}{\beta} \mrm{TI}(\varepsilon_0, \tau)
\end{equation*}
which entails a contradiction: the height of any such cut-free proof must be $\geq \varepsilon_0$. 
Suppose, on the other hand, that $s=t$ is false. Then $\mc{D}_0$ would be a cut-free, SK-proof from $Y$ of $\mrm{TI}(\varepsilon_0, \tau)$. Since rules in the SK cut-free proofs are invertible, a breaking-down of the subformulae in $\mrm{TI}(\varepsilon_0, \tau)$ reveals that there must be cut-free, SK proofs from $Y$ of $\T\corn{\tau(\eta)}$ or $\neg \T\corn{\tau(\eta)}$, $\eta$ some ordinal $<\varepsilon_0$. But this can only happen if $\tau(\eta)\in Y$ or $\neg \tau(\eta)\in Y$. This is a contradiction.

\emph{Case 2}: $B$ is of the form $\T t$ or $\neg \T t$. Suppose that the former is the case. The latter can be established by a symmetric argument. Then, $\mc{D}_1$ yields a proof in $\pat_\omega$ of $\mrm{TI}(\varepsilon_0, \tau), \neg\T t$ of length $\beta<\varepsilon_0$. Additionally, $\mc{D}_0$ is a cut-free $\mrm{SK}$-proof from $Y$ of $\mrm{TI}(\varepsilon_0, \tau), \T t$, and we also know that there is no SK cut-free proof of $\mrm{TI}(\varepsilon_0, \tau)$. Hence, there must be an cut-free, SK-proof from $Y$ of $\T t$, so that $t^\nat \in Y$. We can also show, by induction on the length of a $\pat_\omega$-proof, that
\begin{itemize}
\item[($\star$)] For $\vphi\in \lt$, and $t$ a closed term: if $\mathcal{D}$ is a $\pat^\omega$ cut-free derivation witnessing  $\pat^\omega \sststile{0}{\alpha} \vphi, \neg\T t$, then: either $\T(s)$, for $s^\nat = t^\nat$, is a subformula of $\vphi$, or there is a $\pat^\omega$ cut-free derivation $\mathcal{D}'$ witnessing $\pat^\omega \sststile{0}{\alpha} \vphi$.
\end{itemize}
Since $\mc{D}_1$ is a proof in $\pat_\omega$ of $\mrm{TI}(\varepsilon_0, \tau), \neg\T t$ of length $\beta<\varepsilon_0$, $(\star)$ implies that: either $\T(s)$, for $s^\nat = t^\nat$, is a subformula of $\mrm{TI}(\varepsilon_0, \tau)$; or there is a $\pat^\omega$ cut-free derivation of $\mrm{TI}(\varepsilon_0, \tau)$ {of length $\beta$.}
Now it's easy to see that each disjunct leads to a contradiction. The latter we have already ruled out. For the former, we note that $\T t$ must then be a formula of the form $\T\corn{\tau(r)}$ or $\neg \T\corn{\tau(r)}$ for some term $r$. But this would mean that $\tau(r)^\nat \in Y$ or $\neg\tau(r)^\nat \in Y$, which are impossible.

So, the assumption that $\mrm{TI}(\tau,\varepsilon_0)\in \mrm{I}_\mrm{ssk}$ leads to a contradiction. Finally, assuming $\mrm{Prog}(\tau)$ and that $\varepsilon_0=\{\mrm{sup}(\omega_n)\sth n\in \omega\}$, an application of the $\omega$-rule would enable one to go from $\forall x (x < \omega_{\bar n}\ra\tau(x)))$ for each $n$ to $\forall y,x(x<f(y)\ra\tau(x))$, which would entail the consequent of $\mrm{TI}(\tau,\varepsilon_0)$, hence a contradiction.

\end{proof}

It's worth noting that the `$\mrm{SK}$-fragment' of the construction is still closed under the $\omega$-rule, in the sense that, if each of the $A(\bar{n})$ are $\mrm{SK}$-satisfied in the relevant parameter, so is $\forall x A$.

We also note that nothing rests on the choice of the truth-teller in the previous proposition. For instance, we could have used a parametrized liar sentence or other `ungrounded' sentences. In fact, we can obtain a more general result. Inspecting the proof of the above Proposition, we only employ the minimality property of the fixed point to indicate that truth-tellers cannot be in the fixed point. But, since it is clear that parametrized Liar sentences $\lambda(x)$ cannot be in \textit{any} consistent fixed point, it immediately follows that: 

\begin{prop}\label{big failure of omega logic}
For any consistent $X=\mrm{SSK}$, there are formulae $\vphi(x)$ of $\lt$ such that $\# \vphi(\bar{n})\in X$ for any $n\in \omega$, but $\#\forall y\, \vphi\notin X$. 
\end{prop}

The issues above also imply that Stern's theory $\indt$ is not an $\nat$-categorical axiomatization of the fixed points of either SSK or $\Theta$. In particular: 

\begin{prop}\label{failures wrt it}
There are sets of codes of sentences $S$ such that: 

\begin{enumerate}
    \item\label{Tout} $S=\Theta(S)$ and $S\nvDash \forall t_1,...,t_n(\T\ulcorner \vphi(\subdot t_1,...,\subdot t_n)\urcorner \ra \vphi(t_1,...,t_n))$ for some formula $\vphi(t_1,...,t_n)$ of $\lt$.
    \item\label{Tnegdel} $S=\Theta(S)$ and $S\nvDash \forall t(\T (\subdot \neg \subdot \T t)\ra \T(\subdot{\neg}t) \vee \neg \mrm{Sent}(t))$
    \item $S=\mrm{SSK}(S)$ and $S\nvDash \forall v,x (\mrm{Sent}(\subdot{\forall}vx)\rightarrow (\forall t\T(x\, ( t/v))\rightarrow\T(\subdot{\forall} v x)))$
    \item $S=\mrm{SSK}(S)$ and $S\nvDash \forall x,y(\mrm{Sent}(x\subdot{\vee}y)\rightarrow (\T (x)\vee \T (y)\vee \exists z(\xi(z, \T)\wedge\mrm{Pr}_{\pat}(z\subdot \rightarrow x\subdot{\vee}y ))\ra \T(x\subdot{\vee}y)))$
    \item $S=\mrm{SSK}(S)$ and $S\nvDash\forall x,y(\mrm{Sent}(\subdot{\exists}vx)\rightarrow (\exists t\T(x\, (t/v))\vee \exists w(\xi(w, \T)\wedge\mrm{Pr}_{\pat}(w\subdot{\rightarrow} \subdot{\exists}vx ))\ra \T( \subdot{\exists}vx)))$ 
\end{enumerate}
Hence, no theory featuring $\T$-$\mrm{Out}$ or $\mrm{IT7}$ (right-to-left) can be $\nat$-categorical with respect to fixed points of SSK, and no theory featuring $\mrm{IT4}$, $\mrm{IT3}$, or $\mrm{IT5}$ can be $\nat$-categorical with respect to fixed points of $\Theta$. 
\end{prop}
\begin{proof}
(1) follows from Proposition \ref{failures of theta}.1. (2) follows from Proposition \ref{failures of theta}.2. (3) follows from Proposition \ref{failure of omega logic}. (4) follows by letting $z$ be the code for the formula $\mrm{TI}(\tau,\varepsilon_0)$ as given in Proposition \ref{failure of omega logic}, and $x\subdot \vee y$ be the code for $\mrm{TI}(\tau,\varepsilon_0)\vee \mrm{TI}(\tau,\varepsilon_0)$: by inspecting the proof of that proposition, it is easy to see that $(\nat,S)\vDash \xi(z, \T)$. Finally, (5) follows, just like (4), by letting $w$ be the code of $\mrm{TI}(\tau,\varepsilon_0)$ and $\subdot \exists v x$ be the code for $\exists x_0 (\mrm{TI}(\tau,\varepsilon_0) \wedge x_0>0)$. 
\end{proof}

Although, by \ref{Tout} and \ref{Tnegdel} of Proposition \ref{failures wrt it}, $\Theta$ cannot be categorically axiomatized by $\mrm{IT}$, some fixed points of $\Theta$ are still models for $\mrm{IT}$. In particular, all fixed points that are classically sound, $\T$-downwards closed, and $\neg \T$-downwards closed, will be models of $\mrm{IT}$ (cf.~Proposition \ref{prop:theta-star_n-cat-it}). Since these notions will play a prominent role in what follows, we provide precise definitions for them. 
\begin{dfn}\label{dfn:classound_downnwards_closure}
Let $S\subseteq \omega$, then 
\begin{enumerate}
\item $S$ is \emph{classically sound} iff $(\nat,S)\vDash\varphi$ for all $\#\varphi\in S$;
\item $S$ is \emph{$\T$-downwards closed} iff $\#\T\corn{\varphi}\in S$ implies $\#\varphi\in S$;
\item $S$ is  \emph{$\neg \T$-downwards closed} iff $\#\neg\T\corn\varphi\in S$ when $\#\neg\varphi\in S$.
\end{enumerate}
\end{dfn}
\noindent Proposition \ref{failures of theta} showed that some $\Theta$ fixed points are not $\T$-downwards closed, and hence not classically sound, and some are not $\neg\T$-downwards closed. This being said, some $\Theta$ fixed points meet all the conditions: they are classically sound, $\T$- and $\neg\T$-downwards closed. One such fixed point is the minimal fixed point.

\begin{prop}\label{prop:itheta_models_it}
$\mrm{I}_\Theta\vDash \indt$.
\end{prop}
\begin{proof}[Proof]

We define the Tait language $\lt^\textsc{Tait}$ as follows: given a formula  $\vphi$ of $\lt$, we take the negation-normal form of $\vphi$, i.e., negations are only allowed on literals, and double negations are eliminated from literals. Hence, any formula of $\lt^\textsc{Tait}$ is the result of taking literals or negated literals and applying the connectives $\wedge, \vee, \forall, \exists$. We can safely assume that $\lt$ and $\lt^\textsc{Tait}$ are extensionally equivalent, modulo the definition of negation. 

Now take the system $\mrm{SV}_\infty$, the infinitary (one-sided) sequent calculus in the language $\lt^\textsc{Tait}$ given by the relation $\sststile{\rho}{\alpha}\Gamma$, whose axioms and rules are displayed in Table \ref{tab:svinf}.
\begin{table}[t]
\rule{\linewidth}{0.4pt}
\begin{align*}
&\sststile{\rho}{\alpha}\Gamma,\varphi,\text{for $\varphi$ a literal of $\lnat$} \tag{Ax. 1};\\
&\sststile{\rho}{\alpha}\Gamma, \T (t),\neg \T (s), \text{ with $\# s^\nat=\# t^\nat$} \tag{Ax. 2}\\
&\AxiomC{$\sststile{\rho}{\alpha_0}\Gamma,\varphi_i, \varphi_0\vee \varphi_1$ \text{ for $i\in\{0,1\}$}}\RightLabel{($\vee$)}
    \UnaryInfC{$\sststile{\rho}{\alpha}\Gamma,\varphi_0\vee \varphi_1$}
        \DisplayProof
&&
    \AxiomC{$\sststile{\rho}{\alpha_0}\Gamma,\varphi_0, \varphi_0\wedge \vphi_1$}\RightLabel{($\vee$)}
    \AxiomC{$\sststile{\rho}{\alpha_0}\Gamma,\varphi_1, \varphi_0\wedge \varphi_1$}\RightLabel{($\wedge$)}
    \BinaryInfC{$\sststile{\rho}{\alpha}\Gamma,\varphi_0\wedge \varphi_1$}
        \DisplayProof
\\
&
    \AxiomC{$\sststile{\rho}{\alpha_0}\Gamma,\exists x \varphi, \varphi(\ovl{n})\;\;\text{for some $n\in\omega$}$}\RightLabel{($\exists$)}
    \UnaryInfC{$\sststile{\rho}{\alpha}\Gamma,\exists x \varphi$}
        \DisplayProof
&&
    \AxiomC{$\sststile{\rho}{\alpha_i}\Gamma,\forall x \varphi, \varphi(\ovl{n})\;\;\text{for all $n\in\omega$ and some $\alpha_i<\alpha$}$}\RightLabel{($\omega$)}
    \UnaryInfC{$\sststile{\rho}{\alpha}\Gamma,\forall x \varphi$}
        \DisplayProof
\\
    &
        \AxiomC{$
        \sststile{\rho_0}{\alpha_0} \varphi$}\RightLabel{{\sc $\T$-intro}}
            \UnaryInfC{$\sststile{\rho}{\alpha}\Gamma,\T\corn{\varphi}$}
                \DisplayProof
    &&
        \AxiomC{$\sststile{\rho_0}{\alpha_0}\neg \varphi$}\RightLabel{{\sc $\neg\T$-intro}}
            \UnaryInfC{$\sststile{\rho}{\alpha}\Gamma,\neg \T\corn{\varphi}$}
                \DisplayProof
\end{align*}
\rule{\linewidth}{0.4pt}
\caption{$\mrm{SV}_\infty$}\label{tab:svinf}
\end{table}
We write $\mrm{SV}_\infty\vdash \vphi$ iff there are $\alpha,\rho\in \mrm{On}$ s.t. $\sststile{\rho}{\alpha}\vphi$. Cantini showed that the following two rules are admissible, given the invertibility properties of the system: 
\begin{align*}
    &
        \AxiomC{$
        \mrm{SV}_\infty\vdash\T\corn{\vphi}$}\RightLabel{{\sc $\T$-Elim}}
            \UnaryInfC{$ \mrm{SV}_\infty\vdash \vphi$}
                \DisplayProof
    &&
        \AxiomC{$\mrm{SV}_\infty\vdash\neg \T\corn{\vphi}$}\RightLabel{{\sc $\neg\T$-Elim}}
            \UnaryInfC{$ \mrm{SV}_\infty\vdash\neg \vphi$}
                \DisplayProof
\end{align*}
One can then show that, for any formula $\vphi$, $\mrm{SV}_\infty\vdash \vphi \Leftrightarrow \#\vphi\in \mrm{I}_\Theta$. The left-to-right direction is proved in \cite[Lemma 4.4]{stern_2018}. The right-to-left direction is proven by induction on the stages of the construction of $\mrm{I}_\Theta$, with a subinduction on the positive complexity of the $\vphi$.  Finally, by induction on the length of the derivation in $\mrm{SV}_\infty$, one can show that $\mrm{SV}_\infty\vdash \vphi$ implies $(\nat, \{\psi \sth \mrm{SV}_\infty\vdash \psi\})\vDash \vphi$. This suffices to prove that $(\nat, \mrm{I}_\Theta)$ is a model for axioms IT6, IT7 and $\T$-Out. Verifying the rest is immediate. 
\end{proof}

\section{The $\nat$-categoricity of SSK}\label{sec:n-cat ssk}

One major difficulty in achieving an $\nat$-categorical axiomatization for $\mrm{SSK}$-fixed points is the inexpressibility, in first-order logic, of the notion of (Strong Kleene) satisfaction built in its definition. Indeed, the Strong Kleene satisfaction relation over the standard model of arithmetic, just like the classical one, is hyperarithmetical;\footnote{For the proof that the classical satisfaction relation is hyperarithmetical, see \cite[Theorem 5B.2]{moschovakis_1974} and \cite[Chapter 5]{mcgee_1991}. The proof carries over to the Strong Kleene case with only minor modifications.} as is well-known, this means that it cannot be defined by a formula of $\lnat$, or even $\lt$. This does not exclude, however, that there may be ways of capturing enough of Strong-Kleene satisfaction to obtain an extensionally correct characterization of $\mrm{SSK}$-fixed points. In this section we provide one such axiomatization.

 By the diagonal lemma, $\pat$ proves that there is a formula $\pi(x)$ such that: 
\begin{align*}
    \pi(x)\lra & \mrm{True}_0(x)\vee\\
            & \exists t( x=\subdot \T t\land \T\val{t})\,\vee\\
             & \exists t( x=\subdot \neg\subdot \T t\land (\T\subdot \neg \val{t}\vee \neg \mrm{Sent}_\T(t^\circ))\,\vee\\
             &\exists y (x=\subdot \neg\subdot \neg y \land \T\corn{\pi(\dot{y})})\,\vee\\
             &\exists y,z(x=(y\subdot\land z)\land \T\corn{\pi(\dot{y})})\land \T\corn{\pi(\dot{z})})\\
             &\exists y,z(x=\subdot \neg(y\subdot \land z)\land \T\corn{\pi(\subdot \neg\dot{y})})\land \T\corn{\pi(\subdot \neg\dot{z})})\\
             &\exists y (x=\subdot \forall v y \land \forall t\,\T\corn{\pi(\dot{y}(\subdot t/v)})\,\vee\\
             &\exists y (x=\subdot \neg \subdot \forall v y \land \forall t\,\T\corn{\pi(\subdot \neg\dot{y}(\subdot t/v)}).
\end{align*}

Thus, the formula $\pi(x)$ performs a standard iteration for truth ascriptions, but exploits the diagonal lemma to yield Strong Kleene compositionality in the case of logically complex formulae -- in much the same way that the well-known definition of $\mrm{KF}$ over the theory $\mrm{PUTB}$ does (see, e.g., \cite[Ch.19]{halbach_2014}). 

The following result shows that, in a certain class of models, $\pi(x)$ captures Strong-Kleene satisfaction:

\begin{lemma}\label{putb and eta}
Any $S\subseteq \mrm{Sent}_\T$ with $(\nat,S)\vDash \forall x (\T\corn{\pi(\dot x)}\lra \pi(x))$ is such that 
\begin{equation*}
    S\vDash \pi(\corn{\vphi}) \,\text{ iff }\,(\nat,S, S')\vDash_{\mrm{sk}} \vphi
\end{equation*}
\end{lemma}

\begin{proof}
One needs to induct on the positive complexity of $\vphi$. Suppose that $\vphi$ is an arithmetic literal. Then

\begin{center}
$S\vDash \pi(\corn{\vphi})$ iff $\#\vphi \in \mrm{True}_0$ iff $\nat\vDash \vphi$ iff $(\nat, S, S^-)\vDash_\mrm{sk} \vphi$.
\end{center}
The case for $\vphi$ of the form $\T t$ is also straightforward: 

\begin{center}
$S\vDash \pi(\corn{\T t})$ iff $t^\circ \in S$ iff $(\nat, S, S^-)\vDash_\mrm{sk} \T t$.
\end{center}
When $\vphi$ is of the form $\neg \T t$, the argument is analogous. If $\vphi$ is a logically complex formula, the claim follows almost immediately by IH. Take the case of $\vphi$ being $\neg\neg\psi$: 
\begin{equation*}
    S\vDash \pi(\corn{\neg\neg\psi})\Leftrightarrow S\vDash \T\corn{\pi(\corn{\psi})} \Leftrightarrow S\vDash \pi(\corn{\psi}) \Leftrightarrow(\nat, S, S^-)\vDash_\mrm{sk} \psi \Leftrightarrow (\nat, S, S^-)\vDash_\mrm{sk} \neg \neg \psi 
\end{equation*}
where the third biconditional follows evidently by IH. 
\end{proof}

For future use, we recall that, for $\T$-positive formulae, Strong-Kleene satisfaction and classical satisfaction coincide. 
\begin{lemma}\label{sk and classical for t positive}
For $\vphi$ $\T$-positive: $(\nat,S, S^-)\vDash_\mrm{sk}\vphi$ iff $S\vDash \vphi$.
\end{lemma}
\begin{proof}
We induct on the positive complexity of $\vphi$ again. If $\vphi$ is an arithmetic literal or negated literal, the claim is clear. If $\vphi$ is of the form $\T t$, then $(\nat, S, S^-)\vDash_\mrm{sk}\T t$ iff $t^\circ \in S$ iff $S\vDash \T t$. The case in which $\vphi$ is of the form $\neg \T t$ does not arise, since $\vphi$ would not be $\T$-positive. The inductive case follows immediately from the IH, since the SK satisfaction clauses for all connectives (except negation) and classical satisfaction coincide. 
\end{proof}

\begin{dfn}
Let $\mrm{PK}$ be the theory 
\begin{align}
\tag{$\mrm{TB}\pi$} & \forall x (\T\corn{\pi(\dot x)}\lra \pi(x))\\
\tag{$\mrm{PK}$-$\T$} & \T (x) \lra \exists y (\pi(y)\land \mrm{Pr}_{\pat}(y\subdot \ra x)). 
\end{align}
\end{dfn}
As the following lemma shows, the theory $\mrm{PK}$ has similar compositional properties as the theory $\mrm{IT}$ with the exception that the truth predicate does not commute with the universal quantifier.
\begin{lemma}\label{PK-theorem}
 The theory $\mrm{PK}$ proves the following sentences:
\begin{enumerate}
\item $\forall x(\T(\subdot\T (x))\leftrightarrow\T (x))$;
\item $\forall x(\T(\subdot{\neg}\subdot\T( x)) \leftrightarrow\T(\subdot{\neg} x)\vee\neg\mrm{Sent}(x))$;
\item $\forall x(\mrm{Sent}(x)\ra(\T(\subdot{\neg}\subdot{\neg}x)\leftrightarrow\T (x)))$;
\item $\forall x,y(\mrm{Sent}(x\subdot{\wedge}y)\rightarrow(\T (x\subdot{\wedge}y)\leftrightarrow\T (x)\wedge\T (y)))$;
\item $\forall x,y(\mrm{Sent}(x\subdot{\vee}y)\rightarrow (\T (x)\vee \T (y)\vee \exists z(\pi(z)\wedge\mrm{Pr}_{\pat}(z\subdot \rightarrow x\subdot{\vee}y ))\leftrightarrow \T(x\subdot{\vee}y)))$;
\item $\forall x,v(\mrm{Sent}(\subdot\forall vx)\rightarrow(\T(\subdot{\forall} vx)\rightarrow\forall t\T (x(t/v))))$;
\item $\forall x,y(\mrm{Sent}(\subdot{\exists}vx)\rightarrow (\exists t\T(x\, (t/v))\vee \exists w(\pi(w)\wedge\mrm{Pr}_{\pat}(w\subdot{\rightarrow} \subdot{\exists}vx ))\leftrightarrow \T( \subdot{\exists}vx)))$.
\end{enumerate}
\end{lemma}
\begin{proof}
    As an exemplary case, we take (5). We reason within $\mrm{PK}$. 
    
    For the left-to-right direction, assume $\T (x)\vee \T (y)\vee \exists z(\pi(z)\wedge\mrm{Pr}_{\pat}(z\subdot \rightarrow x\subdot{\vee}y ))$. If $\T(x)$, then $\exists z(\pi(z)\wedge\mrm{Pr}_{\pat}(z\subdot \rightarrow x ))$ by PK-Tr. By logic, then, $\exists z(\pi(z)\wedge\mrm{Pr}_{\pat}(z\subdot \rightarrow x\subdot{\vee}y ))$. So applying PK-Tr again, we get $\T(x\subdot \vee y)$. Similarly for $\T(y)$. Finally, if $\exists z(\pi(z)\wedge\mrm{Pr}_{\pat}(z\subdot \rightarrow x\subdot{\vee}y ))$, $\T(x\subdot \vee y)$ just follows by PK-Tr. 
    
    For the right-to-left direction, assume $\T(x\subdot \vee y)$. By PK-Tr, $\exists z(\pi(z)\wedge\mrm{Pr}_{\pat}(z\subdot \rightarrow x\subdot{\vee}y ))$, whence $\T (x)\vee \T (y)\vee \exists z(\pi(z)\wedge\mrm{Pr}_{\pat}(z\subdot \rightarrow x\subdot{\vee}y ))$ follows by logic.
\end{proof}
Now, we establish the $\nat$-categoricity results for $\mrm{SSK}$ fixed points:

\begin{thm}\label{thm:n-categoricity_ssk}
For $S\subseteq \mrm{Sent}_\T$, and $\mrm{CON}(S)$, $\mrm{SSK}(S)=S$ iff $(\nat, S)\vDash \mrm{PK}$. 
\end{thm}

We first prove (almost) one direction as a lemma: 

\begin{lemma}\label{ssk fps and putb}
Let $S\subseteq \mrm{Sent}_\T$ be such that $\mrm{CON}(S)$ and $S=\mrm{SSK}(S)$. Then, $S\vDash \mrm{TB}\pi$.
\end{lemma}
\begin{proof}
We check that $S\vDash \T\corn{\pi(\ovl{n})}\lra \pi(\ovl{n})$ for all $n\in\omega$. The left-to-right direction follows for any $n$ directly from the fact that SSK satisfies classical soundness, that is, one can immediately verify that $S\vDash_\mrm{ssk}\pi(\ovl{n})\Ra S\vDash\pi(\ovl{n})$ (see also \cite[Lemma 3.7]{stern_2018}). For the right-to-left direction, note that, by lemma \ref{sk and classical for t positive}, $S\vDash\pi(\ovl{n})$ implies $S\vDash_\mrm{sk}\pi(\ovl{n})$, and so $S\vDash_\mrm{ssk}\pi(\ovl{n})$; so, since $S=\mrm{SSK}(S)$, $\#\pi(\ovl{n}) \in S$, and $S\vDash \T\corn{\pi(\ovl{n})}$.
\end{proof}

\begin{proof}[Proof of Theorem \ref{thm:n-categoricity_ssk}]
For the left-to-right direction: lemma \ref{ssk fps and putb} shows that $S\vDash \mrm{TB}\pi$. Now, we want to prove the axiom $\text{($\pk$-$\T$)}$. Assume $S\vDash \T \corn{\vphi}$. Then, since $S=\mrm{SSK}(S)$, there is some $\psi$ such that $S\vDash_\mrm{sk}\psi \; \& \; \mrm{PAT}\vdash \psi\ra \vphi$. By lemma \ref{putb and eta}, and since $(\nat, S)$ is a model of $\mrm{TB}\pi$, $S\vDash_\mrm{sk}\psi$ implies $S\vDash \pi(\corn{\psi})$. Further, since $(\nat, S)$ is a model of $\pat$, $S\vDash \exists y (\pi(y)\wedge \mrm{Pr}_{\pat}(y \subdot \ra \ulcorner \vphi\urcorner))$. The reverse reasoning yields the opposite direction of the axiom. 

For the right-to-left direction: we reason as follows:

\begin{align*}
    \#\vphi \in S & \Leftrightarrow S\vDash \T\corn{\vphi}\\
    & \Leftrightarrow S\vDash \exists \corn{\psi}(\pi(\corn{\psi}\wedge \mrm{Pr}_{\pat}(\corn{\psi\ra \vphi})) & \text{ by ($\pk$-$\T$)}\\
    & \Leftrightarrow S\vDash \pi(\corn{\chi}) \text{ and } S\vDash \mrm{Pr}_{\pat}(\corn{\chi\ra \vphi}) & \text{ for some $\chi\in \mrm{Sent}_{\T}$}\\
    & \Leftrightarrow S\vDash_\mrm{sk} \chi \text{ and } \pat \vdash \chi\ra \vphi & \text{by lemma \ref{putb and eta}}\\
    & \Leftrightarrow \#\vphi \in \mrm{SSK}(S)
\end{align*}
\end{proof}

In a sense, $\pk$ amounts to the bare minimum required to establish its $\nat$-categoricity with respect to $\mrm{SSK}$. We can extend the theory with principles that are independently justified and sound with respect to the intended semantics, and yet retain $\nat$-categoricity. Principles that appear particularly natural are, on the one hand, a less artificial restriction of the naive $\T$-schema allowing disquotation for all $\T$-positive sentences; on the other, since the $\mrm{SSK}$-fixed points over $\nat$ are classically sound, it is plausible to reflect this in the axiomatization by adding the schema $\T\text{-}\mrm{Out}$ to it. 
\begin{dfn}
We call $\mrm{PK}^+$ the theory obtained by replacing, in $\mrm{PK}$,  $\mrm{TB}\pi$ with the schema 
\begin{align}
\tag{$\mrm{PUTB}$} & \T\corn{\vphi(\vec{t})}\lra \vphi(\val{\vec{t}})\,\,\text{ for $\vphi$ $\T$-positive}
\end{align}
and adding the schema $\T\text{-}\mrm{Out}$ to it. 
\end{dfn}
\noindent $\mrm{PK}^+$ is also $\nat$-categorical with respect to $\mrm{SSK}$. The argument above carries over with minor modifications: Lemma \ref{ssk fps and putb} still holds if one requires satisfiability of $\mrm{PUTB}$ instead of just $\mrm{TB}\pi$. This is enough, in the proof of Theorem \ref{PK-theorem}, to show that fixed points of $\mrm{SSK}$ indeed satisfy $\mrm{PUTB}$. 

Since $\mrm{PK}^+$ includes $\mrm{PUTB}$, it is proof-theoretically at least as strong as $\mrm{KF}$ or ramified truth up to $\varepsilon_0$. $\indt$ (see Def.~\ref{dn:ind truth}) is clearly an upper bound for $\pk^+$, but it's evidently not sharp as the former has the same proof-theoretic strength of the theory of inductive definitions $\mrm{ID}_1$ \cite[Cor. 6.7]{stern_2018}. It is an open question whether $\pk^+$ has precisely the strength of $\mrm{PUTB}$. In addition, a full proof-theoretic analysis of $\pk$ is lacking.\footnote{One possibility, suggested by an anonymous reviewer, might be the reduction of $\pk$ to $\widehat{\mrm{ID}_1}$.}

\section{The $\nat$-categorical axiomatization of $\Theta$, $\Theta^*$, and variations}\label{sec:n-cat theta}

\subsection{$\nat$-categorical axioms for $\Theta$}
By Proposition \ref{failures wrt it}, $\mrm{IT}$ isn't $\nat$-categorical with respect to the $\Theta$ fixed points, but are there other axiomatic theories which are $\nat$-categorical with respect to those fixed points? In this subsection, we answer that question in the affirmative. 

\begin{dfn}[$\mrm{IT}^-$]
The theory $\mrm{IT}^-$ extends $\pat$ with the following axioms:
\begin{enumerate}
	\item[$\mrm{IT1}$] $\forall s,t(\T(s\subdot{=}t)\leftrightarrow s^\circ=t^\circ)\wedge \forall s,t(\T(s\subdot{\neq} t)\leftrightarrow s^\circ\neq t^\circ)$
	\item[$\mrm{IT2}$] $\forall x,y(\mrm{Sent}(x\subdot{\wedge}y)\rightarrow (\T (x)\wedge \T (y)\rightarrow \T(x\subdot{\wedge}y)))$
	\item[$\mrm{IT3}$] $\forall x,y(\mrm{Sent}(x\subdot{\vee}y)\rightarrow (\T (x)\vee \T (y)\vee \exists z(\xi(z, \T)\wedge\mrm{Pr}_{\pat}(z\subdot \rightarrow x\subdot{\vee}y ))\leftrightarrow \T(x\subdot{\vee}y)))$
	\item[$\mrm{IT4}$]$\forall v,x (\mrm{Sent}(\subdot{\forall}vx)\rightarrow (\forall t\T(x\, (\dot t/v))\rightarrow\T(\subdot{\forall} v x)))$
	\item[$\mrm{IT5}$]$\forall x,y(\mrm{Sent}(\subdot{\exists}vx)\rightarrow (\exists t\T(x\, (t/v)))\vee \exists w(\xi(w, \T)\wedge\mrm{Pr}_{\pat}(w\subdot{\rightarrow} \subdot{\exists}vx ))\leftrightarrow \T( \subdot{\exists}vx)))$
	\item[$\mrm{IT}^-6$]$\forall x(\T (x)\rightarrow \T(\subdot \T x))$
	\item[$\mrm{IT}^-7$] $\forall t(\T(\subdot{\neg}t) \vee \neg \mrm{Sent}(t)\ra \T (\subdot \neg \subdot \T t) )$
    \item[$\mrm{IT8}$]$\forall x,y(\T (x)\wedge \mrm{Pr}_{\pat}(x\subdot{\rightarrow}y)\rightarrow \T (y))$ 
	\item[$\mrm{IT9}$]$\forall x(\T (\subdot{\neg}x)\rightarrow \neg \T (x))$
	\item[$\mrm{IT10}$]$\T\ulcorner \forall x(\T ( x)\rightarrow \mrm{Sent}( x))\urcorner$
\end{enumerate}	
\end{dfn}

Thus, the major differences between $\mrm{IT}^-$ and $\mrm{IT}$ are simply the exclusion of T-Out and of the right-to-left directions of axioms IT6 and IT7. Note that
Axiom IT9 is provable from axioms IT1, IT2 and IT8: if $\T(\subdot \neg x)$ and $\T(x)$ held, for $x$ the code of some formula, IT2 and IT8 will suffice for proving $\T\corn{0=1}$, which would contradict IT1 if the theory is consistent. 

We shall prove: 

\begin{thm}\label{thm:n-categoricity_theta}
$\mrm{IT}^-$ is $\nat$-categorical with respect to fixed points of $\Theta$, i.e., $S=\Theta(S)\Leftrightarrow S\vDash \mrm{IT}^-$.
\end{thm}

\begin{proof}
For the left-to-right direction: for axioms $\mrm{IT1}$-$\mrm{IT}^-7$ the proof proceeds just like \cite[Lemma 5.6]{stern_2018}. IT8 follows from the fact that fixed points of $\Theta$ are closed under $\pat$-provability. As remarked, $\mrm{IT9}$ follows from $\mrm{IT1}, \mrm{IT2}$ and $\mrm{IT8}$. $\mrm{IT10}$ also follows from the definition of the $\Theta$ jump. 

For the right-to-left direction: we need to show that $S=\Theta(S)$, provided $S\vDash \mrm{IT}^-$. But since we know that $\Theta$ is sound for any choice of $S$, what remains to be shown is $\Theta(S)\subseteq S$. Assume $\#\vphi\in\Theta(S)$. Then,

\begin{equation*}
    \nat \vDash \exists \corn{\psi}(\xi(\corn{\psi},S) \wedge \mrm{Pr}_{\pat} (\corn{\psi\ra\vphi})).
\end{equation*}

We first show $\xi(\corn{\psi}, S)$ implies $\#\psi\in S$, which follows easily by an induction on the positive complexity of $\vphi$. If $\vphi$ is atomic or negated atomic, we make use of $\mrm{IT1}$; if it is of the form $\T t^\circ$ or $\neg \T t^\circ $, we make use of $\mrm{IT}^-6$ and $\mrm{IT}^-7$, respectively; if it is a conjunction, we make use of $\mrm{IT2}$; and so on. Once we have shown that $\xi(\corn{\psi}, S)\Ra \#\psi\in S$, $\#\vphi \in S$ follows by using $\mrm{IT8}$. 
\end{proof}

\subsection{$\Theta^*$ and $\nat$-categorical axioms for it}
Fixed points of $\Theta$ are not very attractive though. The reason is that they are not necessarily \emph{Kripke-sets} in the sense that $\#\vphi\in X \Leftrightarrow \#\T\corn{\vphi} \in X$ and $\#\neg\vphi\in X \Leftrightarrow \#\neg \T\corn{\vphi} \in X$. In particular, as we showed in Proposition \ref{failures of theta}, some fixed points will contain codes of sentences of the form $\T\corn{\vphi}$, while not containing the code of $\vphi$ -- and similarly for $\neg \T\corn{\vphi}$ and $\neg\vphi$. This phenomenon is reflected in the absence of the right-to-left directions of $\mrm{IT}^-6$ and $\mrm{IT}^-7$ in $\mrm{IT}^-$. However, by tweaking the formula $\xi(x,X)$ slightly, we can obtain a class of fixed points that are ($\nat$-categorically) axiomatized by a theory which restores the missing components of the original $\mrm{IT6}$ and $\mrm{IT7}$:
\begin{dfn}
Let $\xi^*(x,X)$ be the formula obtained by supplementing $\xi(x,X)$ from Definition \ref{arithmetical operator} with additional $\T$-downwards closure properties:
    \begin{align*}
& x\in \mathrm{True}_0\\
		&\vee  \exists y,z(x=(y\subdot{\vee} z) \wedge (y\in X\vee z\in X))\\
		&\vee  \exists y,z(x=(y\subdot{\wedge} z) \wedge (y\in X\wedge z\in X))\\
		&\vee  \exists y,v(x=(\subdot{\forall} vy) \wedge \forall t(y (t/v)\in X))\\
		&\vee  \exists y,v(x=(\subdot{\exists} vy) \wedge \exists t(y (t/v)\in X))\\
		&\vee  \exists t(x=\subdot \T (t) \wedge t^\circ \in X)\\
		&\vee  \exists t(x=\subdot \neg \subdot \T (t) \wedge (\subdot{\neg}t^\circ)\in X\vee \neg \mathrm{Sent}(t^\circ))\\
        &\vee  \exists t(x= t^\circ \wedge \mrm{Sent}(t^\circ) \wedge \subdot \T \,(t) \in X)\\
		&\vee  \exists t(x= \subdot \neg t^\circ \wedge \mrm{Sent}(\subdot\neg t^\circ) \wedge \subdot\neg\subdot \T \,(t) \in X)\\
\end{align*}
\end{dfn}

We then define: 
\begin{equation*}
    \Theta^*(S):= \{ n\in\omega \sth \nat \vDash \exists y (\xi^*(y,X) \wedge \mrm{Pr}_{\pat}(y\subdot\ra x))[n,S]\}
\end{equation*}
Note that, for any set $S$, $S$ is sound with respect to $\Theta^*$ -- just like with $\Theta$. However, it follows immediately from the definition of $\Theta^*$ that $\Theta^*$ fixed points are Kripke sets:
\begin{prop}
Let $X=\Theta^*(X)$. Then, $\#\vphi\in X \Leftrightarrow \#\T\corn{\vphi} \in X$ and $\#\neg\vphi\in X \Leftrightarrow \#\neg \T\corn{\vphi} \in X$.
\end{prop}

It is now easy to see what theory characterizes $\Theta^*$ fixed points $\nat$-categorically:
\begin{thm}\label{thm:n-cat_it_star}
Let $\mrm{IT}^*$ be $\mrm{IT}$ minus the axiom schema of T-Out, and replacing $\xi(x,X)$ in axioms $\mrm{IT3}$ and $\mrm{IT6}$ for $\xi^*(x,X)$. Then, $\mrm{IT}^*$ is $\nat$-categorical with respect to fixed points of $\Theta^*$, i.e., $S=\Theta^*(S)\Leftrightarrow S\vDash \mrm{IT}^*$.
\end{thm}

$\Theta^*$ does not impose any requirement on what the truth predicate looks like \textit{internally}. In a way analogous to that of the different supervaluation schemes, however, we can also require that, internally, the truth predicate be consistent, or consistent and complete. This is done by defining new operators:

\begin{dfn}
Let $\mrm{con}(x)$ be the formula $\exists s,t(x=\corn{\neg (\T \subdot s \wedge \T \subdot \neg \subdot t)}\wedge s^\circ=t^\circ)$, and define:
\begin{equation*}
    \Theta^*_{\mrm{c}}(S):= \{ n\in\omega \sth \nat \vDash \exists y ((\xi^*(y,X) \vee \mrm{con}(y)) \wedge \mrm{Pr}_{\pat}(y\subdot\ra x)) [n, S]\}
\end{equation*}
\end{dfn}

\begin{dfn}
Let $\mrm{com}(x)$ be the formula $\exists s,t(x=\corn{\neg \T \subdot s \lra \T \subdot \neg \subdot t}\wedge s^\circ=t^\circ)$, and define:
\begin{equation*}
    \Theta^*_{\mrm{mc}}(S):= \{ n\in\omega \sth \nat \vDash \exists y ((\xi^*(y,X) \vee \mrm{com}(y)) \wedge \mrm{Pr}_{\pat}(y\subdot\ra x))[n, S]\}
\end{equation*}
\end{dfn}

In light of our $\nat$-categoricity result for $\Theta^*$, the following is easily established: 

\begin{prop}
Let $\mrm{IT}^*_\mrm{c}$ be $\mrm{IT}^*$ plus the axiom $\forall x(\T\corn{\neg (\T \dot x \wedge \T \subdot \neg \dot x)})$, and let $\mrm{IT}^*_\mrm{mc}$ be $\mrm{IT}^*$ plus the axiom $\forall x(\T\corn{\neg \T \dot x \lra \T \subdot \neg \dot x})$. Then, $\mrm{IT}^*_\mrm{c}$ is $\nat$-categorical with respect to fixed points of $\Theta^*_\mrm{c}$, and $\mrm{IT}^*_\mrm{mc}$ is $\nat$-categorical with respect to fixed points of $\Theta^*_\mrm{mc}$.
\end{prop}

\subsection{Relationships between fixed points of $\Theta$ and $\Theta^*$}

We start by providing a general result about $\Theta$ and $\Theta^*$, which reveals that their constructions preserve classical soundness and, in the case of $\Theta$, also $\T$- and $\neg\T$-downwards closure: 

\begin{lemma}\label{lem:ttstarind}
The following hold:
\begin{enumerate}
    \item Let $X$ be a classically sound set of sentences, and define:
\begin{align*}
    &\Delta_0:=X\\
    &\Delta_{\alpha+1}:=\Theta^*(\Delta_\alpha)\\
    & \Delta_\lambda:=\bigcup_{\beta<\lambda}\Delta_\beta
\end{align*}
    Then, $\bigcup_{\alpha\in\mrm{On}}\Delta_\alpha$ -- the least $\Theta^*$ fixed point containing $X$ -- is classically sound.
    \item Let $X$ be a classically sound, $\T$- and $\neg\T$-downwards closed set of sentences, and define:  
\begin{align*}
    &\Gamma_0:=X\\
    &\Gamma_{\alpha+1}:=\Theta(\Gamma_\alpha)\\
    & \Gamma_\lambda:=\bigcup_{\beta<\lambda}\Gamma_\beta
\end{align*}
    Then, $\bigcup_{\alpha\in\mrm{On}}\Gamma_\alpha$  is classically sound, and $\T$- and $\neg\T$-downwards closed.
\end{enumerate}
\end{lemma}

\begin{proof}
For (1): we show that $\#\vphi \in \Delta_\alpha $ implies $ \Delta_\alpha\vDash \vphi$. The base case follows from the assumption that $X$ is classically sound.
For the successor case: we can assume that $\#\vphi \in \Delta_\alpha $ implies $ \Delta_\alpha\vDash \vphi$ and prove the claim for $\alpha+1$. Suppose that $\#\vphi \in\Delta_{\alpha+1}$ by $\xi^*(x,X)$. We induct on the positive complexity of $\vphi$: 
\begin{itemize}
    \item If $\vphi$ is in $\mrm{True}_0$, this is immediate. 
    \item If $\vphi$ is of the form $\T \corn{\psi}$, and $\#\psi \in \Delta_\alpha$: since $\Delta_\alpha\subseteq\Delta_{\alpha+1}$, $\psi\in\Delta_{\alpha+1}$, hence $\Delta_{\alpha+1}\vDash \T\corn{\psi}$.
    \item If $\vphi$ is of the form $\neg\T \corn{\psi}$, and $\#\neg\psi\in\Delta_{\alpha}$: then $\#\neg\psi\in\Delta_{\alpha+1}$ and, since $\Delta_{\alpha+1}$ is consistent, $\#\psi\notin\Delta_{\alpha+1}$, hence $\Delta_{\alpha+1}\vDash \neg \T\corn{\psi}$.
    \item If $\vphi$ is a logically complex formula, the claim follows easily by IH. 
\end{itemize}
If $\#\vphi\in\Delta_{\alpha+1}$ because there is $\psi$ such that $\xi^*(\corn{\psi}, \Delta_\alpha)$ holds and $\pat\vdash \psi\ra \vphi$: since $\Delta_{\alpha+1}\vDash \pat$, and we just showed how $\xi^*(\corn{\psi}, \Delta_\alpha)$ implies $\Delta_{\alpha+1}\vDash \psi$, the claim follows. 

For (2): to prove that $\bigcup_{\alpha\in\mrm{On}}\Gamma_\alpha$ is classically sound, we only need to run again the proof for (1) again. To prove the $\T$- and $\neg\T$-downwards closure, we need to verify that $\#(\neg)\T\corn{\vphi}\in\Gamma_\alpha\Ra\#(\neg)\vphi \in \Gamma_\alpha$. But this is immediate for the base case from our assumption that $X$ is $\T$ and $\neg\T$-downwards closed, and follows by IH for the successor case. The argument just given implies that, for any $\alpha$, $\xi^*(\corn{\vphi}, \Gamma_\alpha) \Leftrightarrow \xi(\corn{\vphi}, \Gamma_\alpha)$, whence we can conclude that $\bigcup_{\alpha\in\mrm{On}}\Gamma_\alpha$ is also the least fixed point of $\Theta^*$ containing $X$.  
\end{proof}
\noindent In the constructions above, we say that $X$ is the \textit{starting set} of the construction.

Next, we show that fixed points of $\Theta^*$ are also fixed points of $\Theta$. This is immediate by inspecting the definitions of $\Theta$ and $\Theta^*$:
\begin{lemma}\label{prop:thetastar implies theta}
Let $S=\Theta^*(S)$. Then, $S=\Theta(S)$. Conversely, if $S=\Theta(S)$ and $S$ is $\T$- and $\neg\T$-downwards closed, then $S=\Theta^*(S)$. 
\end{lemma}

By Proposition \ref{failures of theta}, there are fixed points of $\Theta$ that are not $\T$- and $\neg \T$-downwards closed. However, $\T$ and $\neg \T$-downwards closure are still not sufficient to yield the $\nat$-categoricity of $\mrm{IT}$. There could still be complex sentences, such as $\tau_0\vee \tau_1$, within $\Theta^*$ fixed points, without their constituent sentences being there---which would violate $\T$-Out. However, if classical soundness is added as a condition on the fixed points, we obtain a more restricted class of fixed points that $\mrm{IT}$ can $\nat$-categorically axiomatize. 
\begin{prop}\label{prop:theta-star_n-cat-it}
Let $S\subseteq \mrm{Sent}_\T$. Then the following claims are equivalent:
\begin{enumerate}
\item $S=\Theta^*(S) \; \& \; S \; \text{is classically sound}$;
\item  $S=\Theta(S) \; \& \; S \; \text{is classically sound and } \T\text{- and }\neg\T\text{-downwards closed}$;
\item  $S\vDash \indt$.
 \item $S\vDash \indt^* + \T\text{-}\mrm{Out}$
\end{enumerate}
\end{prop}
\begin{proof}
$1\Leftrightarrow  4$: In Theorem \ref{thm:n-cat_it_star}, we proved that $S=\Theta^*(S)$ iff $S\vDash \indt^*$. But if $S$ is classically sound, then $S\vDash \T\corn{\psi}\ra\psi$ for all $\psi$ s.t. $\psi\in\mrm{Sent}_\T$. The reverse also holds: the fixed-point model $(\nat, S)$ satisfies T-Out only if $S$ is classically sound.\\
$2 \Leftrightarrow 3$: $\indt$ is $\indt^-$ plus $\T$-Out and the right-to-left directions of axioms $\mrm{IT6}$ and $\mrm{IT7}$. By Theorem \ref{thm:n-categoricity_theta}, $S=\Theta(S)$ iff $S\vDash \indt^-$. $\T$-Out holds iff $S$ is classically sound; and the right-to-left directions of $\mrm{IT6}$ and $\mrm{IT7}$ hold exactly when $S$ is $\T$- and $\neg\T$-downwards closed. \\
$1\Leftrightarrow 2$: given by Lemma \ref{prop:thetastar implies theta}.
\end{proof}

\subsection{Proof-Theory}
$\mrm{IT}^*$ also enjoys considerable proof-theoretic strength. In particular, we can show that it is at least as proof-theoretically strong as the theory $\mrm{KF}$. The proof-theoretic strength will be obtained by defining ramified truth predicates, so we refer the reader to \cite[Ch.9]{halbach_2014} for a definition of the relevant languages and systems. Instead of using $\mrm{IT}^*$, however, we will provide a more general result by using $\mrm{IT}^-$; since the latter is a subtheory of $\mrm{IT}^*$, the claim will follow immediately.
\begin{prop}\label{prop:lower bound itminus}
$\mrm{IT}^-$ can define the truth predicates of $\mrm{RT}_{<\varepsilon_0}$.
\end{prop}

\begin{proof}
Consider the formula 
\begin{equation*}
    \mrm{D}(x):= (\T x \vee \T \subdot\neg x) \wedge \forall y,z (x=(y\subdot \vee z) \ra \T y \vee \T z) \wedge \forall y,v (x=(\subdot \exists v y) \ra \exists t \T y(t/v))
\end{equation*}
It is immediate to verify that $\mrm{D}(x)$ satisfies the following properties:
\begin{enumerate}
    \item\label{d_property_one} $\mrm{D}(x)\rightarrow \mrm{D}(\subdot\neg x)$
    \item $\mrm{D}(x)\wedge \mrm{D}(y)\rightarrow \mrm{D}(x\subdot \vee y) \wedge \mrm{D}(x\subdot \rightarrow y)$
    \item $\mrm{Sent}_{\T}(\subdot \forall vy) \wedge \forall t (\mrm{D}( y (t/v)))\rightarrow \mrm{D}(\subdot \forall vy)$
    \item $\mrm{D}(x)\rightarrow (\T (\subdot \neg x) \leftrightarrow \neg \T (x))$
    \item$\mrm{D}(x\subdot\vee y)\rightarrow (\T (x\subdot \vee y) \leftrightarrow \T (x) \vee \T (y))$ 
     \item$\mrm{D}(x\subdot\rightarrow y)\rightarrow (\T (x\subdot \rightarrow y) \leftrightarrow \T (x) \rightarrow \T (y))$ 
    \item $\mrm{D}(x)\rightarrow (\T (\subdot \forall v x) \leftrightarrow \forall y \T (x (y/v)))$
\item\label{d_property_eight} $\mrm{D}(x)\rightarrow (\T (x)\leftrightarrow \T (\subdot \T x))$
\end{enumerate}

We now use this formula to adapt a proof by Fujimoto \cite[Lemma 36] {fujimoto_2010}. We start by defining two p.r. functions \textit{à la} Fujimoto. First, for $\beta< \varepsilon_0$, let $h(x, \beta) $ be a p.r. function from $\omega\times\mrm{On}$ into $\omega$ defined by
\[   
h(x,\beta) = 
     \begin{cases}
       x &\quad\text{if } x\in \mrm{Sent}_{<\beta}\\
       \ulcorner 0=1\urcorner &\quad\text{otherwise} \\
     \end{cases}
\]
Second, we define a function $k(x)$ from $\mc{L}_{<\varepsilon_0}$ into $\lt$. Since the function will operate within codes of sentences, it needs to be defined with the aid of the recursion theorem (see \cite[ch.5]{halbach_2014} for a detailed exposition). The function is thus defined by
    \[   
k(x) = 
     \begin{cases}
       x, &\quad\text{if } x\text{ is an atomic formula of $\lnat$}\\
       \ulcorner \T (k\circ h(t,\beta))\urcorner &\quad\text{if } x=\subdot \T_\beta t \text{ and } t \text{ is a closed term}\\
       \ulcorner \neg \T (k(\dot y))\urcorner  &\quad\text{if } x=\subdot \neg y\\
       \ulcorner \T (k(\dot y)) \vee \T (k(\dot z))\urcorner  &\quad\text{if } x=y \subdot \vee z\\
       \ulcorner \T k(\dot y) \wedge \T (k(\dot z))\urcorner  &\quad\text{if } x=y \subdot \wedge z\\
       \ulcorner \forall z \T (k(\dot y (u/\corn{z})))\urcorner(z/\corn{u})  &\quad\text{if } x=\subdot \forall z y \text{ and } z \text{ is a variable}\\
       \ulcorner 0=1\urcorner &\quad\mrm{otherwise} \\ 
     \end{cases}
\]
Using the axioms of $\indt^-$ and properties \ref{d_property_one}-\ref{d_property_eight}, one can show:
\begin{equation*}
    \indt^-\vdash \forall \gamma\leq \beta (\mrm{Sent}_{<\gamma}(x) \rightarrow \mrm{D}(kx))
\end{equation*}
Finally, for any ordinal $\alpha<\varepsilon_0$, define $\sigma_\alpha: \mc{L}_{<\alpha}\longrightarrow \mc{L}_\T$ as: 
\[   
\sigma_\alpha(\vphi) = 
     \begin{cases}
       s=t &\quad\text{if } \vphi=(s = t)\\
       \T(kx) & \quad \text{if } \vphi=\T_\beta(x), \beta<\alpha\\
       \neg \sigma_\alpha(\psi) &\quad\text{if } \vphi=\neg\psi\\
    \sigma_\alpha(\psi) \vee \sigma_\alpha(\chi)&\quad\text{if } \vphi=\psi\vee\chi\\
 \sigma_\alpha(\psi) \wedge \sigma_\alpha(\chi)&\quad\text{if } \vphi=\psi\wedge\chi\\
  \exists v \sigma_\alpha(\psi)&\quad\text{if } \vphi=\exists v \psi\\
  \forall v \sigma_\alpha(\psi)&\quad\text{if } \vphi=\forall v \psi\\
     \end{cases}
\]
It is easy to verify that translating $\T_\alpha(x)$ as $\sigma_\alpha(x)$, for $\alpha<\varepsilon_0$, yields a truth-definition of $\mrm{RT}_{<\varepsilon_0}$ in $\indt^-$.
\end{proof}

\begin{corollary}
$\vert \mrm{IT}^* \vert \geqq \vert \mrm{IT}^- \vert  \geqq \vert \mrm{KF}\vert\equiv\vert \mrm{RT}_{<\varepsilon_0}\vert$
\end{corollary}
\noindent Just like in the case of $\pk$,  $\mrm{IT}^*, \mrm{IT}^-$ are sub-theories of $\mrm{IT}$. However, a sharp proof-theoretic analysis of $\mrm{IT}^*, \mrm{IT}^-$ is not yet available. 

By contrast, the proof-theoretic analysis of $\mrm{IT}^*+\T\text{-Out}$ follows the one for $\mrm{IT}$ from \cite{stern_2018}: the lower bound is provided by the provability of Bar-Induction in (a fragment of) $\mrm{IT}^*+\T\text{-Out}$, which is known to be sufficient to derive the arithmetical theorems of $\mrm{ID}_1$; for the upper bound argument, the strategy is to formalize, in the theory $\mrm{KPU}$, the stages of the construction of the minimal fixed point of $\Theta$, and establish the equivalence between membership in the fixed point and provability in the infinitary system defined in Proposition \ref{prop:itheta_models_it}.\footnote{For the lower-bound, Stern follows the proof-theoretic analysis in \cite{friedman_sheard_1987}; for a similar analysis, which Cantini carries out for the theory $\mathrm{VF}$, see \cite{cantini_1990}. The upper-bound follows the strategy in Cantini's paper.} Since, provably in $\mrm{KPU}$, $\mrm{I}_{\Theta}=\mrm{I}_{\Theta^*}$, we can use such a provability predicate as an adequate interpretation for the truth predicate of $\mrm{IT}^*+\T\text{-}\mrm{Out}$. That is:
\begin{prop}\label{prop:ptaittout}
$\vert \mrm{IT}^*+\T\text{-}\mrm{Out}\vert= \vert \mrm{IT} \vert = \vert \mrm{ID}_1 \vert = \vert \mrm{VF} \vert$ 
\end{prop}

\section{Relations with standard supervaluationist fixed points}\label{sec:relssksup}

In this section we establish some connections between supervaluation-style fixed-point semantics and the standard supervaluational approaches. The aim is to draw a precise picture of the prospects and limitations of supervaluation-style semantics and its connection to standard supervaluational semantics.

Recall the schema $\mrm{VB}$ introduced in Section \ref{supervaluations}. It is easy to see that the minimal fixed points of $\Theta$, $\Theta^*$ and $\mrm{VB}$ coincide:

\begin{corollary}
$\mrm{I}_\Theta=\mrm{I}_{\Theta^*} =\mrm{I}_\mrm{vb}$.
\end{corollary}
\begin{proof}
The proof of Proposition \ref{prop:itheta_models_it} shows that $\#\vphi \in \mrm{I}_\Theta \Leftrightarrow \mrm{SV}_\infty \vdash \vphi$. But Cantini \cite[\S 4, esp. Thm. 4.7]{cantini_1990} proved that $\#\vphi \in \mrm{I}_\mrm{vb} \Leftrightarrow \mrm{SV}_\infty \vdash \vphi$. It then follows that $\#\vphi \in \mrm{I}_\Theta \Leftrightarrow \#\vphi \in \mrm{I}_\mrm{vb} $.

We now show that $\#\vphi \in \mrm{I}_\Theta \Leftrightarrow \#\vphi \in \mrm{I}_{\Theta^*} $. We know that VB fixed points are $\T$- and $\neg\T$-downwards closed (Def. \ref{dfn:classound_downnwards_closure}); combining this with the fact that $\mrm{I}_\Theta$ is a fixed point of $\Theta$, it follows that it is also a fixed point of $\Theta^*$. But if $\mrm{I}_{\Theta}$ were not the minimal fixed point of $\Theta^*$, there would be a set $S=\Theta^*(S)$ such that $S\subsetneq \mrm{I}_{\Theta}$. Also, by Lemma \ref{prop:thetastar implies theta}, $S=\Theta(S)$. But this contradicts the fact that $\mrm{I}_\Theta$ is the $\Theta$-least fixed point.
\end{proof}

By contrast, it can be shown that the minimal fixed point of SSK is a proper subset of the least $\mrm{VB}$ fixed point---and, as a consequence, of the least fixed points of $\Theta$ and $\Theta^*$:

\begin{prop}\label{prop:ssk and vb}
$\mrm{I}_\mrm{ssk}\subsetneq \mrm{I}_\mrm{vb}$. 
\end{prop}
\begin{proof}
That $\mrm{I}_\mrm{ssk}\subseteq \mrm{I}_\mrm{vb}$ was already proven in \cite[Lemma 3.6]{stern_2018}: indeed, the implication 
\begin{align*}
\vphi \in \mrm{SSK}(S) & \Ra  \forall S'\supseteq S (\Phi_\mrm{vb}(S')\Ra S'\vDash \vphi)
\end{align*}
holds because, as we have seen, (i) satisfaction in Strong Kleene entails $\mrm{VB}$-satisfaction, and (ii) $\mrm{VB}$-truth is closed under $\mrm{PAT}$-consequence. 
%
That the inclusion is proper follows from Proposition \ref{failure of omega logic}, where we showed how $\mrm{I}_\mrm{ssk}$ fails to be closed under $\omega$-logic. By contrast, as is well-known, VB fixed points (including $\mrm{I}_\mrm{vb}$) are closed under $\omega$-logic.
\end{proof}



It is clear that some fixed points of $\Theta$ are not VB fixed points. For instance, some fixed points of $\Theta$ fail to be transparent for the truth predicate (Proposition \ref{failures of theta}). It is also clear that some $\Theta^*$ fixed points are not VB fixed points: all fixed-point models of VB satisfy T-Out, but not all $\Theta^*$ fixed points do.\footnote{Again, just take the fixed point that obtains by applying $\Theta^*$ to $\{\#\tau_1\vee\tau_2\}$.} What if we restrict our attention to $\Theta^*$ fixed-point models of T-Out, i.e., that are classically sound? In Proposition \ref{prop:theta-star_n-cat-it}, we showed that those models are precisely the ones that satisfy $\indt$. It then follows that the class of classically sound fixed points of $\Theta^*$ cannot coincide with the class of supervaluational, VB fixed points. Otherwise, VB would have an $\nat$-categorical axiomatization (namely, $\indt$), which we know isn't the case by the complexity considerations described in the introduction.

There is more we can say about fixed points of VB and of $\Theta^*$. First, that the class of VB fixed points is a subclass of the $\Theta^*$ fixed points: 

\begin{lemma}\label{every fp of vb is a fp of thetastar}
Let $S=\mrm{VB}(S)$. Then, $S=\Theta^*(S)$, and $S=\Theta(S)$, and $S$ is classically sound.
\end{lemma}
\begin{proof}
Assume $S=\mrm{VB}(S)$. We want to show  $S=\Theta^*(S)$. $S\subseteq \Theta^*(S)$ follows immediately, since we know that every set of formulae is $\Theta^*$- sound. For the other direction, we induct on the clauses of $\xi^*(x,X)$, and employ the VB fixed-point property of $S$ in order to prove that $\xi^*(x, S)\Ra x\in S$. It is easy to see why this holds: any fixed point of VB is closed under $\omega$-logic (and so it is closed under the clauses for the connectives and quantifiers in $\xi^*$), and is closed under $\T$-Intro, $\T$-Elim, $\neg \T$-Intro and $\neg \T$-Elim, and so closed under the $\T$-clauses. Furthermore, VB fixed points are also closed under $\pat$-provability, so $x\in\Theta^*(S)\Ra x\in S$. 

Having shown $S=\Theta^*(S)$, that $S=\Theta(S)$ just follows from Lemma \ref{prop:thetastar implies theta}. Also, that $S$ is classically sound follows immediately, since $S$ is a VB-fixed point. 
\end{proof}

In fact, we can show that any least fixed point generated by a $\mrm{VB}$-sound set corresponds to the least fixed point of $\Theta$ generated from that set. To prove that, we need two observations. 

\begin{lemma}\label{lemma: vb soundness preserved}
Let $S\subseteq \mrm{Sent}_\T$ and $S\subseteq \mrm{VB}(S)$. Then, $\Theta(S)\subseteq \mrm{VB}(S)$.
\end{lemma}
\begin{proof}
The lemma can be proved by first checking 
\begin{equation*}
    \xi(\corn{\vphi}, X) \Ra \#\vphi\in \mrm{VB}(X)
\end{equation*}
whenever $X$ is $\mrm{VB}$-sound. The closure of the $\mrm{VB}$ operator under $\mrm{PAT}$-consequence yields the claim. 

\end{proof}
\begin{lemma}\label{prop:pat prov implies theta fp}
Given a set $S\subseteq \mrm{Sent}_\T$, we write $\pat\cup \{ \T\corn{\psi} \sth \#\psi \in S\}\vdash^\omega \vphi$ iff $\vphi$ is provable in $\omega$-logic from $\pat$ and each sentence of the form $\T\corn{\psi}$, $\#\psi\in S$, as axioms. Now, let $X=\Theta(X)$. Then:
\begin{equation*}
\pat\cup \{ \T\corn{\psi} \sth \#\psi \in X\}\vdash^\omega \vphi \Ra \#\vphi \in X
\end{equation*}
\end{lemma}
\begin{proof}
This is immediate by induction on the length of the ($\omega$)-derivation of $\vphi$ in  $\pat\cup \{ \T\corn{\psi} \sth \#\psi \in S\}$. 
The axioms of $\pat$ are given by the $\theta$ formula, and the axioms of the form $\T\corn{\psi}$, $\#\psi\in X$, follow from the fixed-point property. Since $\Theta$ fixed points are closed under modus ponens and $\omega$-logic, the claim follows.
\end{proof}

\begin{prop}\label{lemma:vb_sound_least_theta}
Let $X\subseteq \mrm{VB}(X)$. Consider the following constructions: 
\begin{align*}
& \Gamma^{\mrm{vb}}_0= X & & \Gamma^{\Theta}_0= X\\
& \Gamma^{\mrm{vb}}_{\alpha+1}= \mrm{VB}(\Gamma^{\mrm{vb}}_{\alpha}) & & \Gamma^{\Theta}_{\alpha+1}= \Theta(\Gamma^{\Theta}_{\alpha})\\
& \Gamma^{\mrm{vb}}_\lambda=\bigcup_{\beta<\lambda} \Gamma^{\mrm{vb}}_\beta, \text{ for $\lambda$ a limit}& & \Gamma^{\Theta}_\lambda=\bigcup_{\beta<\lambda} \Gamma^{\Theta}_\beta, \text{ for $\lambda$ a limit}\\
\end{align*}
Then, $\bigcup_{\beta\in \mrm{On}} \Gamma^{\mrm{vb}}_\beta=\bigcup_{\beta\in \mrm{On}} \Gamma^{\Theta}_\beta$.
\end{prop}
\begin{proof}
For clarity of notation, let $\bigcup_{\beta\in \mrm{On}} \Gamma^{\Theta}_\beta = S$. We have that $\Theta(S)=S$, and also $\mrm{VB}(S)\supseteq \Theta(S)$. We assume, for reductio, that $S\neq \bigcup_{\beta\in \mrm{On}} \Gamma^{\mrm{vb}}_\beta$. By Lemma \ref{lemma: vb soundness preserved}, we know that $\Theta(X)\subseteq \mrm{VB}(X)$; and, in fact, for any $\alpha\in \mrm{On}$, $\Theta^\alpha(X)\subseteq \mrm{VB}(\Theta^\alpha (X))$ (where $\Theta^\alpha$ stands for $\alpha$-th iterations of $\Theta$). Therefore, $S\subseteq \mrm{VB}(S)$. But then $S\neq \bigcup_{\beta\in \mrm{On}} \Gamma^{\mrm{vb}}_\beta$ implies  $\mrm{VB}(S)\supsetneq S$. So there must be some sentence $\vphi$ s.t. $\#\vphi \in \mrm{VB}(S)$ but $\#\vphi \notin \Theta(S)=S$. By Lemma \ref{prop:pat prov implies theta fp}, we obtain: 
\begin{equation}\label{eqn:omega comp}
\pat\cup \{ \T\corn{\psi} \sth \#\psi \in S\}\nvdash^\omega \vphi
\end{equation}
By the completeness theorem for $\omega$-logic, $\pat\cup \{ \T\corn{\psi} \sth \#\psi \in S\}\vdash^\omega \vphi$ holds iff for all models $(\nat, X)$ satisfying all $\T\corn{\psi}$ with $\#\psi \in S$, $(\nat, X)\vDash \vphi$. Hence, (\ref{eqn:omega comp}) implies that there must be a standard model $(\nat, S')$, with $S'\supseteq S$, s.t. $(\nat, S')\nvDash\vphi$. But then $\#\vphi\notin \mrm{VB}(S)$, yielding a contradiction. 
\end{proof}

Let us now examine in greater depth the connection between the class of classically sound $\Theta^*$ fixed points and the $\mrm{VB}$ fixed points. First, let us denote: 
\begin{equation*}
\mc{Z}:=\{ S \sth S=\Theta^*(S) \; \& \; S \; \text{is classically sound}\}
\end{equation*}
We will consider this class of fixed points and refer to them in the usual terms. In particular, we say that some fixed point $S$ is \emph{$\mc{Z}$-maximal} iff there is no $S'\in \mc{Z}$ s.t. $S\subsetneq S'$. By adapting a result from \cite{stern_2018} to the present setting, we can show that the $\mc{Z}$-maximal fixed points and the $\mrm{VB}$ maximal fixed points are not the same. 

\begin{dfn}
Let $\Phi\colon \mc{P}(\omega)\to \mc{P}(\omega)$ be an operator. 
Recall that $X^-:=\{\#\neg\vphi \sth \#\vphi \in X\}\cup \{n\in\omega \sth \nat \vDash \neg\mrm{Sent}_\T(\bar n)\}$ and $Y^-:=\{\#\neg\vphi \sth \#\vphi \in Y\} \cup \{n\in\omega \sth \nat \vDash \neg\mrm{Sent}_\T(\bar n)\}$. We define:
\begin{itemize}
    \item $X,Y$ are \emph{compatible} iff $X\cap Y^-=\varnothing$.
    \item $X$ is \emph{$\Phi$-intrinsic} (in short: intrinsic) iff $X\subseteq \Phi(X)$ and $\mrm{CON}(X)$ and for all $Y$ s.t. $Y\subseteq \Phi(Y)$, $X$ and $Y$ are compatible.
 \end{itemize}
The set $\{ n \in\omega \sth \exists X ( X \text{ is intrinsic } \& \; n \in X)\}$ is known as the \emph{maximal intrinsic fixed point} of $\Phi$. A formula $\vphi$ is said to be $\Phi$-\emph{intrinsic} iff its code is in the maximal intrinsic fixed point of $\Phi$.
\end{dfn}
Notice that for the familiar operators considered in this paper, the existence of the maximal intrinsic fixed point follows from the existence of the relevant minimal fixed points, as these minimal fixed points are themselves intrinsic. 

\begin{lemma}\label{different max intrin fps}
The $\mc{Z}$-maximal intrinsic fixed point and the maximal intrinsic fixed point of $\mrm{VB}$ differ. Therefore, 
\begin{equation*}
    \{ S \sth S\in \mc{Z} \; \& \; S \text{ is maximal}\}\neq \{ S \sth S=\mrm{VB}(S) \; \& \; S \text{ is maximal}\}.
\end{equation*}
\end{lemma}

\begin{proof}
The first claim follows by complexity considerations. The upper-bound for the $\mc{Z}$-maximal intrinsic fixed point ($\mrm{I}^+_{\mc{Z}}$) can be obtained as follows: 
\begin{equation}\label{eqn:complexity iplusz}
    n \in \mrm{I}^+_{\mc{Z}} \Leftrightarrow \exists X (X=\Theta^*(X) \; \& \; X \text{ is classically sound} \; \& \; \forall m\in X (\subdot \neg m \notin \mrm{V}^+_{\Theta^*})  \; \& \; n\in X)
\end{equation}
Here, $\mrm{V}^+_{\Theta^*}$ stands for the set of sentences true in some $\Theta^*$ fixed point. Since the formula $\xi^*(x,X)$ of the $\Theta^*$ jump is arithmetical, $X=\Theta^*(X)$ is a $\Delta_1^1$ formula, as the next definition shows: 
\begin{equation*}
    X=\Theta^*(X) \Leftrightarrow \forall n(n \in X \Leftrightarrow \exists y(\xi^*(y,X)\wedge \mrm{Pr}_{\pat}(y\subdot\ra n))
\end{equation*}
It follows that $\mrm{V}^+_{\Theta^*}$ has complexity at most $\Sigma^1_1$:
\begin{equation*}
    n \in V^+_{\Theta^*} \Leftrightarrow \exists X(X=\Theta^*(X) \wedge n\in X)
\end{equation*}
Finally, the property of classical soundness is at most $\Delta^1_1$:
\begin{equation*}
    X \text{ is classically sound } \Leftrightarrow \forall \vphi (\#\vphi \in X \Ra X\vDash \vphi) 
\end{equation*}
where we rely on the well-known fact that classical satisfaction is $\Delta^1_1$.

Following (\ref{eqn:complexity iplusz}), membership in $\mrm{I}^+_{\mc{Z}}$ is at most $\Sigma^1_1$-in-a-($\omega\setminus\Sigma^1_1$)-parameter, i.e., $\Sigma^1_1$-in-a-$\Pi^1_1$-parameter. By contrast, Burgess \cite{burgess_1988} showed that the maximal intrinsic fixed point of the VB scheme ($\mrm{I}^+_\mrm{vb}$) is precisely $\Delta^1_2$-in-a-$\Pi^1_2$-parameter, so $\mrm{I}^+_\mrm{vb}\neq\mrm{I}^+_\mc{Z} $.

The `therefore' part of the claim follows from the fact that the maximal intrinsic point of a scheme or operator is the intersection of all maximal fixed points of the operator---see e.g. \cite[Th.2.18.2]{visser_1984}.
\end{proof}

In fact, we can provide specific examples of simple fixed points that are in $\mc{Z}$ but are not $\mrm{VB}$ fixed points. The following is a variation of an example we have used repeatedly:

\begin{lemma}\label{lemma:example of different vb-theta fp}
Consider the following construction:

\begin{align*}
    &\Gamma_0:=\{\#(\neg\tau_1\vee\neg\tau_2)\}\\
    &\Gamma_{\alpha+1}:=\Theta^*(\Gamma_\alpha)\\
    & \Gamma_\lambda:=\bigcup_{\beta<\lambda}\Gamma_\beta
\end{align*}
Then, $\bigcup_{\alpha\in\mrm{On}}\Gamma_\alpha\in \mc{Z}$ but $\bigcup_{\alpha\in\mrm{On}}\Gamma_\alpha\neq \mrm{VB}(\bigcup_{\alpha\in\mrm{On}}\Gamma_\alpha)$.
\end{lemma}
\begin{proof}
We first show that $\bigcup_{\alpha\in\mrm{On}}\Gamma_\alpha$ is classically sound. We prove that $\#\vphi \in \Gamma_\alpha $ implies $ \Gamma_\alpha\vDash \vphi$ by transfinite induction on $\alpha$. 

For the base case: it is clear that $\Gamma_0\vDash \neg \tau_1\vee\neg \tau_2$, since $\neg\tau_i\equiv \neg \T\corn{\tau_i}$ and $\#\tau_i\notin \Gamma_0$ ($i$ one of $1,2$). 
For the successor case: we can assume that $\#\vphi \in \Gamma_\alpha $ implies $ \Gamma_\alpha\vDash \vphi$ and prove the claim for $\alpha+1$. If $\#\vphi \in\Gamma_{\alpha+1}$ by $\xi^*(x,X)$, the claim follows by an induction on the positive complexity of $\vphi$. It's worth noting that in the case in which $\vphi$ is of the form $\T \corn{\psi}$ we employ the increasing nature of $\Theta^*$, whereas if $\vphi$ is of the form $\neg\T \corn{\psi}$ we employ the consistency of $\Gamma_\alpha$ for any $\alpha$. 

On the other hand, if $\#\vphi\in \Gamma_{\alpha+1}$ because there is some $\psi$ such that $\xi^*(\corn{\psi}, \Gamma_\alpha)$ holds and $\psi\ra\vphi$ is provable in $\pat$, the previous argument shows $\Gamma_{\alpha+1}\vDash \psi$. But since this is a model of $\pat$, $\Gamma_{\alpha+1}\vDash \vphi$.


Now, we show that $\bigcup_{\alpha\in\mrm{On}}\Gamma_\alpha\neq \mrm{VB}(\bigcup_{\alpha\in\mrm{On}}\Gamma_\alpha)$. It is clear that $\#(\neg\tau_1\vee\neg\tau_2)\in \bigcup_{\alpha\in\mrm{On}}\Gamma_\alpha$ and $\#\neg\tau_1\notin \bigcup_{\alpha\in\mrm{On}}\Gamma_\alpha$, $\#\neg\tau_2\notin \bigcup_{\alpha\in\mrm{On}}\Gamma_\alpha$. But no consistent fixed point of VB can contain $\#(\neg\tau_1\vee\neg\tau_2)$ and fail to contain $\#\neg\tau_1$ or $\#\neg\tau_2$. Since $\#\neg\tau_1\notin \bigcup_{\alpha\in\mrm{On}}\Gamma_\alpha$ and $\#\neg\tau_2\notin \bigcup_{\alpha\in\mrm{On}}\Gamma_\alpha$, there are extensions $Y\supseteq \bigcup_{\alpha\in\mrm{On}}\Gamma_\alpha$ such that both $\#\tau_1\in Y$ and $\#\tau_2\in Y$, whence $(\nat, Y)\vDash \neg (\neg\tau_1\vee\neg\tau_2)$. But then, $\#(\neg\tau_1\vee\neg\tau_2)\notin\mrm{VB}(\bigcup_{\alpha\in\mrm{On}}\Gamma_\alpha)$, so $\bigcup_{\alpha\in\mrm{On}}\Gamma_\alpha\neq \mrm{VB}(\bigcup_{\alpha\in\mrm{On}}\Gamma_\alpha)$.
\end{proof}

\begin{prop}
$\{ S \sth S=\Theta^*(S) \; \& \; S \; \text{is classically sound}\} \neq \{ S\sth S=\mrm{VB}(S)\}$
\end{prop}
\begin{proof}
Lemma \ref{different max intrin fps} yields the result. Note that one can also use Lemma \ref{lemma:example of different vb-theta fp}, reasoning as follows. By that lemma, either (i) there is a $\mc{Z}$-maximal fixed point $X$ which is not a $\mrm{VB}$-maximal fixed point, or (ii) there is a $\mrm{VB}$-maximal fixed point $Y$ which is not a $\mc{Z}$-maximal fixed point. If (i): then either $X$ is not a fixed point of $\mrm{VB}$, in which case the claim follows; or there is some $X'\supseteq X$ which is a fixed point of $\mrm{VB}$, and cannot be a fixed point of $\mc{Z}$, since $X$ is $\mc{Z}$-maximal, whence the claim follows. An analogous reasoning applies if (ii) is the case. 
\end{proof}

The situation is very similar when we consider the operator $\Theta_\mrm{c}$ and the supervaluationist scheme VC. We omit the proofs of the following results, which are just reproductions of all proofs we have carried out for VB:

\begin{prop}\label{consprop}
The following hold:

\begin{enumerate}
    \item Let $S=\Theta_\mrm{c}^*(S)$. Then, $S=\Theta_\mrm{c}(S)$.
    \item $\mrm{I}_{\Theta_\mrm{c}}=\mrm{I}_{\Theta^*_\mrm{c}}=\mrm{I}_{\mrm{vc}}$.
    \item $\mrm{I}_{\mrm{ssk}_\mrm{c}}\subsetneq \mrm{I}_{\mrm{vc}}$
    \item  $S=\Theta^*_\mrm{c}(S) \; \& \; S \; \text{is classically sound} \; \Leftrightarrow S\vDash \indt_\mrm{c}$.
    \item Let $\mc{Z}_\mrm{c}=\{ S\sth \Theta^*_\mrm{c}(S) \; \& \; S \; \text{is classically sound}\}$. Then, $\mc{Z}_\mrm{c}\supseteq \{ S\sth S=\mrm{VC}(S)\}$.
    \item The $\mc{Z}_\mrm{c}$-maximal intrinsic fixed point and the maximal intrinsic fixed point of $\mrm{VC}$ differ.
    \item Let $Y_\mrm{c}$ be the $\Theta^*_\mrm{c}$ fixed point obtained by the construction of Lemma \ref{lemma:example of different vb-theta fp}, with $\Theta^*_\mrm{c}$ in place of $\Theta^*$. Then, $Y_c\in \mc{Z}_\mrm{c}$ but $Y_c\neq \mrm{VC}(Y_c)$.
\end{enumerate}

\end{prop}

In the case of MC and $\Theta^*_\mrm{mc}$, it must be noted that the scheme MC is not itself classically sound. Therefore, we can forget about that aspect when establishing their relations. Thus, the connection is even more direct than the case of $\Theta^*$ and $\Theta^*_\mrm{c}$, although the proofs do not change much:

\begin{prop}
The following holds:

\begin{enumerate}
    \item Let $S=\Theta_\mrm{mc}^*(S)$. Then, $S=\Theta_\mrm{mc}(S)$.
    \item $\mrm{I}_{\Theta_\mrm{mc}}=\mrm{I}_{\Theta^*_\mrm{mc}}=\mrm{I}_{\mrm{mc}}$.
     \item $\mrm{I}_{\mrm{ssk}_\mrm{mc}}\subsetneq \mrm{I}_{\mrm{mc}}$
    \item  $S=\Theta^*_\mrm{mc}(S) \Leftrightarrow S\vDash \indt_\mrm{mc}$.
    \item If $S=\mrm{MC}(S)$, then $S=\Theta^*_\mrm{mc}(S)$.
    \item The $\Theta^*_\mrm{mc}$-maximal intrinsic fixed point and the maximal intrinsic fixed point of $\mrm{MC}$ differ.
\end{enumerate}
\end{prop}

Yet something makes the case of $\Theta^*_\mrm{mc}$ particularly different: the example given in Lemma \ref{lemma:example of different vb-theta fp} does not apply here. The reason is simple: when one considers \textit{maximally consistent} extensions of a set $Y$ containing $\#(\neg\tau_1\vee\neg\tau_2)$, the maximal consistency requirement implies that at least one of the disjuncts ($\#\neg\tau_1$ or $\#\neg\tau_2$) must be in the extension as well. 

This being said, here are a few more things we can uncover about where MC and $\Theta^*_\mrm{mc}$ come apart. As we did before with $\Theta^*$ and $\Theta^*_\mrm{c}$, we display a concrete counterexample of a fixed point of $\Theta^*_\mrm{mc}$ which is not a fixed point of MC.

\begin{prop}\label{prop:example_mc_thetamc}
Let $\lambda_1$, $\lambda_2$ be two Liar sentences, i.e., $\pat\vdash \lambda_i\lra \neg\T\corn{\lambda_i}$ (for $i\in\{1,2\}$). Consider the following construction. 

\begin{align*}
    &\Gamma_0:=\{\#(\lambda_1\lra\lambda_2)\}\\
    &\Gamma_{\alpha+1}:=\Theta^*_\mrm{mc}(\Gamma_\alpha)\\
    & \Gamma_\lambda:=\bigcup_{\beta<\lambda}\Gamma_\beta
\end{align*}
Then, $\bigcup_{\beta\in\mrm{On}}\Gamma_\beta\neq \mrm{MC}(\bigcup_{\beta\in\mrm{On}}\Gamma_\beta)$.
\end{prop}
\begin{proof}
First, it must be noted that $\bigcup_{\beta\in\mrm{On}}\Gamma_\beta$ is indeed a consistent $\Theta^*_\mrm{mc}$ fixed point: we know that every consistent set of sentences (as is $\{\#(\lambda_1\lra\lambda_2)\}$) generates a consistent $\Theta^*_\mrm{mc}$ fixed point. Moreover, as before, $\{\#(\lambda_1\lra\lambda_2)\}\subseteq \bigcup_{\beta\in\mrm{On}}\Gamma_\beta$. Secondly, we check that no fixed point of MC can contain $\#(\lambda_1\lra\lambda_2)$. The reasoning is as follows. It is easy to verify that no fixed point can contain either $\lambda_1$ or $\lambda_2$. But then, if $X$ were a fixed point containing $\#(\lambda_1\lra\lambda_2)$, there would be, e.g., a (maximally-consistent) extension $X'\supseteq X$ for which $\#\lambda_1\in X'$ and $\lambda_2\notin X'$. Then, $(\nat, X')\nvDash \lambda_1\lra\lambda$, so $\mrm{MC}(X)\neq X$. 
\end{proof}

\begin{corollary}
There is a $\Pi^1_1$ fixed point of $\Theta^*_\mrm{mc}$ which is not a fixed point of $\mrm{MC}$.
\end{corollary}
\begin{proof}
The fixed point constructed in the proof of Proposition \ref{prop:example_mc_thetamc} suffices. The (positive) parameter in the induction, $\{\#(\lambda_1\lra\lambda_2)\}$, is arithmetical. Hence, the fixed point can be described by a $\Pi^1_1$-formula. 
\end{proof}

\begin{remark}
The construction:

\begin{align*}
    &\Gamma_0:=\{\#(\lambda_1\lra\lambda_2)\}\\
    &\Gamma_{\alpha+1}:=\Theta^*_\mrm{e}(\Gamma_\alpha)\\
    & \Gamma_\lambda:=\bigcup_{\beta<\lambda}\Gamma_\beta,
\end{align*}
\noindent for $\Theta^*_\mrm{e}$ one of $\Theta^*, \Theta^*_\mrm{c}$, also yields a fixed point of the corresponding $\Theta^*$ operator which isn't a fixed point of VB or VC, respectively. This observation draws on \cite{burgess_1986}, who already proved that no supervaluationist fixed point can contain the code of the sentence  $\#(\lambda_1\lra\lambda_2)$.
\end{remark}

\section{Conclusion}\label{sec:conc}

\noindent In the introductory Section \ref{sec:intro}, we have claimed that the \textsc{SSK thesis} can be vindicated, albeit in a less uniform manner than was envisaged by \cite{stern_2018}. In the light of the results just presented, we now explain in more detail what the vindication and the non-uniformity consists in, and sketch their philosophical import. 


We can say that the \textsc{SSK thesis} has been vindicated on the following grounds. First, all the operators we have considered -- $\mrm{SSK}$, $\Theta$, and $\Theta^*$ -- give rise to a fixed-point semantics that accounts for all first-order penumbral truths of classical logic and arithmetic. That is, they preserve what in Section \ref{sec:intro} we dubbed the \emph{essential aspects} of supervaluational truth. Second, the fixed points in question have a considerably lower recursion-theoretic complexity, and are therefore amenable to $\nat$-categorical axiomatizations, as we have shown. {Moreover, the $\mrm{SSK}$-fixed points are true to the semantic motivation supporting the supervaluation schemes we discussed in Section \ref{sec:intro}: the $\mrm{SSK}$-scheme collects the first-order logical consequences of Strong Kleene truths. In contrast, the $\Theta$- and $\Theta^*$-fixed points merely collect consequences of a given hypothesis. Yet, the flipside is that attractive $\Theta$ ($\Theta^*$)-fixed points are fixed points of supervaluational truth, as they actually collect consequences in $\omega$-logic. $\mrm{SSK}$-fixed points are not closed under $\omega$-logic and will thus fall short of supervaluational truth, that is, $\mrm{SSK}$-fixed points are not supervaluational fixed points.


 We see that each of these incarnations of supervaluational-style truth has its own strengths and weaknesses. We now analyze them in turn along key dimensions: 

 \subsection*{Principles of Truth.} Although the fixed-point models of $\mrm{SSK}$, $\Theta$, and $\Theta^*$ have all $\nat$-categorical axiomatizations, this does not entail that all such axiom systems are equally natural. 
 
 The principles of truth for $\mrm{IT}^-$, $\mrm{IT}^*$ and $\mrm{IT}^*+\T\text{-}\mrm{Out}$ preserve the desirable features of the the theory $\mrm{IT}$ discussed in \cite{stern_2018}. In particular. they are arguably as compositional as supervaluational-style truth can be. In fact, following Stern's characterization of $\mrm{IT}$ in \cite[p.~841]{stern_2018}, they have the advantage that, when full compositionality fails, we can locate where it does and the extent of this failure by means of a first-order formula -- $\xi(x,X)$ for $\mrm{IT}^-$, $\xi^*(x,X)$ for $\mrm{IT}^*$ and $\mrm{IT}^*+\T\text{-}\mrm{Out}$. But even between these theories 
 important differences arise. As we have seen, the lack of a right-to-left direction in axioms $\mrm{IT}^-6$ and $\mrm{IT}^-7$, which is remedied in $\mrm{IT}^*$ and $\mrm{IT}^*+\T\text{-}\mrm{Out}$, makes the latter more natural than $\mrm{IT}^-$: unlike $\mrm{IT}^-$, $\mrm{IT}^*$ and $\mrm{IT}^*+\T\text{-}\mrm{Out}$ allow for full truth disquotation inside the truth predicate. In addition, $\mrm{IT}^*+\T\text{-}\mrm{Out}$ stands out for its proof-theoretic strength, which matches the one of $\mrm{IT}$. They amount to the strongest, $\nat$-categorical axiomatic theories of truth available. 

The truth predicate of the $\mrm{IT}$-theories commutes with the universal quantifier, in contrast to the truth predicate of the theories of $\mrm{PK}$ and $\mrm{PK}^+$, as the $\Theta$-fixed points, unlike the $\mrm{SSK}$-fixed points, are closed under the $\omega$-rule. While there are good reasons why the $\mrm{SSK}$-fixed points are not closed under $\omega$-logic, it means that the theories $\mrm{PK}$ and $\mrm{PK}^+$ cannot give a satisfactory semantic explanation for the truth of universally quantified sentences. Putting universally quantified sentences aside, $\mrm{PK}$ and $\mrm{PK}^+$ preserve the desirable features of theories of supervaluation-style truth as highlighted by Lemma \ref{PK-theorem}:
\begin{itemize}
\item they allow for full disquotation inside the truth predicate;
\item they allow for double negation introduction/elimination inside the truth predicate
\item they commute with conjunction;
\item they locate the failure of compositionality by means of the formula $\theta(x)$ (for disjunctions, existential sentences and universal sentences).
\end{itemize}
Moreover, from the semantic perspective $\mrm{PK}$ and $\mrm{PK}^+$ are more successful in limiting the failure of compositionality since  $\T\corn{\neg\tau_1\vee\neg\tau_2}$ will be true in the models of $\mrm{PK}$ and $\mrm{PK}^+$ only if either $\T\corn{\neg\tau_1}$ or $\T\corn{\neg\tau_2}$ (cf.~Lemma \ref{lemma:example of different vb-theta fp}) is true in the model.

However, the theory $\mrm{PK}$ assumes an, arguably, fairly {\it ad hoc} axiom, $\mrm{TB}\pi$, which seems difficult to motivate. The axiom is required to guarantee that we can express $S\vDash_\mrm{sk}\varphi$ over the standard model. Introducing $\mrm{TB}\pi$ enables us to define this notion by a positive inductive definition (relative to the standard model). Without this axiom we would not be able to show the $\nat$-categoricity of $\mrm{PK}$ and the theory would not behave in the intended way. 
$\mrm{PK}^+$ remedies this by allowing disquotation for all $\T$-positive sentences. Arguably, this is a more natural restriction to the naive $\T$-schema, as it is justified by the role played by negation (or equivalent logical tools) in the Liar paradox. In addition, $\pk^+$ features the axiom schema of $\T$-Out, which naturally follows from the classical soundness of $\mrm{SSK}$-fixed points.




\subsection*{Semantic Rationale.}
 We have already mentioned that the $\mrm{SSK}$ fixed-point semantics comes equipped with semantico-philosophical rationale. For one, the fixed-point semantics is associated with a semantic evaluation scheme, the $\mrm{SSK}$-scheme, which, arguably, captures directly the core idea of `first-orderizing' the `second-order' supervaluation schema. For another, but for related reasons, the $\mrm{SSK}$-fixed points are submitted to semantic scrutiny: every sentence in a given  $\mrm{SSK}$-fixed point $S$ must be true in the model $(\nat,S)$ according to the $\mrm{SSK}$-scheme, that is, $\mrm{SSK}$-fixed points are semantically self-supporting. Since the $\mrm{SSK}$-scheme is classically sound this also implies that the members of $S$ are classically true in the model $(\nat,S)$.
 
 Neither holds for the fixed-point semantics associated with the operations $\Theta$ or $\Theta^*$: there is no motivated semantic evaluation scheme associated with these fixed-point semantics, nor are the fixed points subjected to any form of semantic scrutiny. The latter shows in the fact that $\Theta$ and $\Theta^*$ are sound for any given hypothesis $S\in\mathcal{P}(\mrm{Sent})$, that is, starting from any set of sentences $S$ we can inductively define a fixed point $\mrm{I}_S\supseteq S$ of $\Theta$ and $\Theta^*$ respectively. This means that the sentences in the starting hypothesis are accepted as true without scrutiny independently of their semantic status. Unsurprisingly then, not all fixed points of $\Theta$ and $\Theta^*$ are classically sound -- and, even if they are classically sound, they may contain, e.g., disjunctions that are neither supported by the Strong Kleene scheme nor $\mrm{PAT}$-truths (see the example of Lemma \ref{lemma:example of different vb-theta fp} discussed above). So, generally speaking, the $\Theta$ and $\Theta^*$ fixed points lack semantic justification. Rather, these fixed points can be conceived of as the set of theorems derivable within the proof system outlined in Proposition \ref{prop:itheta_models_it} with $S$ as additional axioms. Ultimately, this means that, from a semantic perspective, the fixed points of $\Theta$ and $\Theta^*$ will be well supported and motivated if(f) the initial hypothesis is. However, motivating a particular hypothesis is external to the fixed-point semantics of $\Theta$ and $\Theta^*$. Yet, if such a motivation for an initial hypothesis $S$ is given, e.g., that it is supervaluationally self-supporting, then this motivation is indeed fully transferred to the fixed point $\mrm{I}_S$ constructed over $S$: $\mrm{I}_S$ will be a supervaluational fixed point. Indeed, as we showed in Proposition \ref{lemma:vb_sound_least_theta}, all supervaluational fixed points can be constructed using $\Theta$ and $\Theta^*$ in such a way.

\subsection*{Two ways of first-orderizing}
Summing up, the $\mrm{SSK}$- and $\Theta$-operators can be considered as two different ways of `first-orderizing' the supervaluation scheme. The $\mrm{SSK}$-operator and scheme provide us with a first-order formulation of the semantic idea underlying the supervaluation scheme. In contrast to the supervaluation scheme, it {\it only} collects the first-order penumbral truths and falls short of collecting second-order penumbral truths. This also explains why the $\mrm{SSK}$-fixed points are not closed under the $\omega$-rule. The $\mrm{SSK}$-scheme embraces the idea of the supervaluation scheme but implements it in an independent way that leads to a distinct (and disjoint) class of supervaluation-style fixed points and a novel supervaluation-style theory.

In contrast, the various $\Theta$-operators `first-orderize' the supervaluational fixed-point construction, as these operators are defined via a first-order arithmetical formula. We can then use the first-order construction to extract axiomatic truth theories that are as close as possible to a proof-theoretic characterization of supervaluational truth (cf.~Proposition \ref{prop:theta-star_n-cat-it}). Yet, in contrast to the $\mrm{SSK}$-scheme, there is no semantics in the ordinary sense associated with these operators, and the fixed-point construction does not lend any semantic justification to the sentences in the fixed points. The fixed points of the various $\Theta$-operators are somewhat `hit-and-miss', depending on the choice of the initial hypothesis. The moral of our discussion then is that we can first-orderize supervaluational truth either by giving a first-order formulation of the underlying semantic idea or by first-orderizing the fixed-point construction. However, these are two different projects that, unfortunately, cannot be performed in an integrated and uniform manner.

\section{Open problems}

After our investigation, there remain some open problems that may be addressed in future research. They concern the proof-theoretic strength of many of the theories that we have introduced here.

\begin{op}
What is the proof-theoretic strength of $\pk$? And of $\pk^+$?
\end{op}

By definition, we know that $\pk^+$ has at least the proof-theoretic strength of $\mrm{PUTB}$, but it is still open whether this lower-bound is precise. By contrast, both the lower- and the upper-bound of $\pk$ are still undetermined.

\begin{op}
What is the proof-theoretic strength of $\mrm{IT}^-$? And $\mrm{IT}^*$?
\end{op}

Proposition \ref{prop:lower bound itminus} established a lower-bound for both $\mrm{IT}^-$ and $\mrm{IT}^*$, namely the strength of $\mrm{RT}_{<\varepsilon_0}$. As with $\pk^+$, it remains unknown whether this bound is exact; while the strength of $\mrm{IT}$ represents an upper-bound, it is clearly not sharp.

\bibliography{biblio1}
\bibliographystyle{alpha}
\end{document}